\documentclass[11pt,a4paper,reqno]{amsart}
\usepackage[applemac]{inputenc}
\usepackage[T1]{fontenc}
\usepackage{hyperref}
\usepackage{amsmath}
\usepackage{amsthm}
\usepackage{amsfonts}
\usepackage{amssymb}
\usepackage{graphicx}
\usepackage{color}
\usepackage{amsbsy}
\usepackage{mathrsfs}
\usepackage{bbm}
\usepackage{xcolor}
\usepackage{ulem}
\usepackage{soul}
\usepackage{dsfont}
\usepackage{graphicx}
\usepackage{subcaption}
\newtheorem{theorem}{Theorem}[section]
\newtheorem{lemma}[theorem]{Lemma}

\newtheorem{proposition}[theorem]{Proposition}
\newtheorem{corollary}[theorem]{Corollary}

\theoremstyle{remark}
\newtheorem{remark}[theorem]{\it \bf{Remark}\/}
\newenvironment{acknowledgement}{\noindent{\bf Acknowledgement.~}}{}
\numberwithin{equation}{section}
\catcode`@=11
\def\section{\@startsection{section}{1}\z@{1.5\linespacing\@plus\linespacing}{.5\linespacing}{\normalfont\bfseries\large\centering}}
\catcode`@=12
\newcommand{\be}{\begin{equation}}
	\newcommand{\ee}{\end{equation}}
\newcommand{\bea}{\begin{eqnarray}}
	\newcommand{\eea}{\end{eqnarray}}
\newcommand{\bee}{\begin{eqnarray*}}
	\newcommand{\eee}{\end{eqnarray*}}
 \newcommand{\la}{\left\langle}
\newcommand{\ra}{\right\rangle}

\def\pa{\partial}

\def\na{\nabla}

\def\NN{\mathbb{N}}
\def\RR{\mathbb{R}}
\def\ZZ{\mathbb{Z}}

\def\eps{\vare}

\def\ep{\varepsilon}

\def\calB{{\mathcal B}}

\def\calE{{\mathcal E}}
\def\calF{{\mathcal F}}

\def\calL{{\mathcal L}}
\def\calM{{\mathcal M}}

\def\calQ{{\mathcal Q}}
\def\calR{{\mathcal R}}

\catcode`@=11
\def\supess{\mathop{\operator@font Sup\,ess}}
\catcode`@=12

\def\NN{\mathbb{N}}

\def\RR{\mathbb{R}}

\def\ZZ{\mathbb{Z}}

\def\e{\varepsilon}

\def\Dein{\Delta^{-1}}

\def\eps{\varepsilon_s}
\def\epu{\varepsilon_u}

\newcommand{\norm}[1]{\left\lVert#1\right\rVert}

\begin{document}
	
	\title[]{Finite time blowup for Keller-Segel equation with logistic damping in three dimensions}

\author[J. Liu]{Jiaqi Liu}
\address{Department of Mathematics, University of Southern California, Los Angeles, California 90089, USA}
\email{jiaqil@usc.edu}

\author[Y. Wang]{Yixuan Wang}
 \address{Applied and Computational Mathematics, California Institute of Technology, Pasadena, California 91125, USA}
\email{roywang@caltech.edu}

\author[T. Zhou]{Tao Zhou}
 \address{Department of Mathematics\\
National University of Singapore\\
Singapore\\
119076\\
Singapore}
\email{zhoutao@u.nus.edu}

	\maketitle

    \begin{abstract}
    The Keller-Segel equation, a classical chemotaxis model, and many of its variants have been extensively studied for decades. In this work, we focus on $3$D Keller-Segel equation with a quadratic logistic damping term $-\mu \rho^2$ (modeling density-dependent mortality rate) and show the existence of finite-time blowup solutions with nonnegative density and finite mass for any $\mu \in \big[0,\frac{1}{3}\big)$. This range of $\mu$ is sharp; for $\mu \ge \frac{1}{3}$, the logistic damping effect suppresses the blowup as shown in \cite{Kang_Stevens_globalexistence_p=2_and_mu_critical,Tello_Winkler_globalexi_for_p=2}.
    A key ingredient is to construct a self-similar blowup solution to a related aggregation equation as an approximate solution, with subcritical scaling relative to the original model.
    Based on this construction, we employ a robust weighted $L^2$ method to prove the stability of this approximate solution, where modulation ODEs are introduced to enforce local vanishing conditions for the perturbation lying in a singular-weighted $L^2$ space. As a byproduct, we exhibit a new family of type I blowup mechanisms for the classical $3$D Keller-Segel equation.
    \end{abstract}
    
    \section{Introduction}
Chemotaxis is a widespread natural phenomenon, and it occurs when organisms, such as body cells or bacteria, detect and move toward chemical signals in their surroundings. A principal mathematical description of chemotaxis is provided by the Keller-Segel system:
\be
\begin{cases}
\partial_t \rho  = \Delta \rho - \na \cdot (\rho \na c), \\
\Delta c + \rho = 0,
\end{cases}
\tag{KS}
\label{equation, Keller-Segel}
\ee
 where $\rho$ represents the density of the bacteria and $c$ denotes the concentration of the self-emitted chemical substance. The model captures two key processes: the diffusive effect of random bacterial motion and the directed movement of bacteria toward the highest concentration of the chemical. For a broader introduction to chemotaxis, see \cite{MR3932458, MR2013508, MR2073515}.

 Biologically, beyond the competition between diffusion and bacterial aggregation, the limitation of resources and overcrowding necessitate introducing the bacterial mortality rate that further suppresses the aggregation process. In particular, we consider the following $3$D coupled parabolic-elliptic Keller-Segel system with logistic damping
\be
\begin{cases}
\partial_t \rho = \Delta \rho -  \na \cdot (\rho \na c) - \mu \rho^2, \\
\Delta c +\rho = 0,
\end{cases}
\label{equation: KS with damping}
\tag{KS-D}
\ee
where $-\mu \rho^2$ ($\mu \ge 0$)  represents the logistic damping rate. For additional background on \eqref{equation: KS with damping} and related models, we refer the interested readers to \cite{Fuest_blowup_optimal_2021,MR2448428,MR1788983,MR540951,Tello_Winkler_globalexi_for_p=2}.

        \subsection{Background}

    \subsubsection{Classical Keller-Segel equation: global existence v.s. finite-time blowup}

\mbox{}

\vspace{0.1cm}

    For the classical Keller-Segel equation \eqref{equation, Keller-Segel}, the total mass $M(t) = \int \rho(t,x) dx$ is conserved for all time.  
Additionally, the system \eqref{equation, Keller-Segel} admits a one-parameter family of scaling invariance: 
\be
\rho(t,x)\to \frac{1}{\lambda^2} \rho\left(\frac{t}{\lambda^2}, \frac{x}{\lambda} \right), \ c(t,x)\to c\left( \frac{t}{\lambda^2}, \frac{x}{\lambda} \right), \quad \forall \; \lambda>0.
\label{scaling invariance}
\ee
We remark that this scaling invariance also holds for \eqref{equation: KS with damping}. 

In two dimensions, the total mass $M=8\pi$ serves as a crucial quantity determining the global existence and finite-time blowup of the system \eqref{equation, Keller-Segel}. Specifically,  it has been shown that finite-time blowup occurs whenever $M> 8 \pi$ with initial data $\rho_0 \in L_+^1((1+|x|^2),dx)$, while global-in-time existence with a uniform $L^\infty$ bound holds when $M<8\pi$, see \cite{Blanchet_Dolbeault_Perthame_globalexistence06, Dolbeault_Perthame_globalexistence04}.

Nevertheless, the critical mass threshold does not exist for the three-dimensional model \eqref{equation, Keller-Segel}. In fact, even with a tiny total mass, there is a radial solution that develops a finite-time singularity found by Nagai \cite{MR1361006}. Beyond the radial case, Corrias, Perthame, and Zaag \cite{MR2099126} demonstrated that blowup occurs whenever the second moment of the initial density is sufficiently small in comparison to the total mass, while weak solutions exist globally if the initial density has a suitably small $L^\frac{3}{2}$ norm. Interested readers can refer to \cite{MR4201903, MR3411100, MR3438649, MR3936129} for more results.

\subsubsection{Classical Keller-Segel equation: singularity formation}
\label{subsub: singularity formation of KS}

\mbox{}

\vspace{0.1cm}

Over the last several decades, the singularity formation for the classical Keller-Segel equation \eqref{equation, Keller-Segel} has been well studied. For the $2$D case, since the model is in the $L^1$ critical sense, Naito and Suzuki \cite{Naito_Suzuki_typeIIblowup} verified that any finite-time blowup solution to \eqref{equation, Keller-Segel} is of type II\footnote{The solution of \eqref{equation, Keller-Segel} exhibits type I blowup at $t=T$ if 
\begin{equation*}
    \limsup_{t \to T} (T-t) \| \rho(t) \|_{L^\infty} < \infty.
\end{equation*}
Otherwise, the blowup is of type II.
}.
In particular, Rapha\"{e}l and Schweyer \cite{Raphael_Schweyer_2DtypeII_blowup14} provided a precise construction of a radially stable finite-time blowup solution of the form
\be
\rho(t,x) \approx \frac{1}{\lambda^2(t)} U\left(\frac{x}{\lambda(t)} \right), \quad U(x) = \frac{8}{(1+|x|^2)^2}, \label{type II blowup: profile}
\ee
where the blowup rate
\[
\lambda(t) = \sqrt{T-t} e^{-\sqrt{\frac{|\ln (T-t)|}{2}}+ O(1)}, \quad t \to T.
\]
Subsequently, this result has been extended to several refined scenarios, including nonradial blowup, multi-bubble blowup without collision, and the simultaneous collision of two collapsing bubbles. For further details, we refer the interested reader to \cite{buseghin2023existence, Collot_Ghoul_Masmoudi_Nguyen_2DtypeII_blowup22, 2DKSmultibubble}.

Unlike the $2$D case, the $3$D Keller-Segel system exhibits a large variety of blowup mechanisms. 
Notably, there is a countable family of radial self-similar blowup solutions of the form
\be
\rho_s(t,x) = \frac{1}{T-t} Q_s\left( \frac{x}{\sqrt{T-t}} \right), \; 
\footnote{It exactly matches the scaling invariance \eqref{scaling invariance} of the classical Keller-Segel system \eqref{equation, Keller-Segel}.}
\label{ansatz: self-similar blowup solution usual scaling}
\ee
which have been identified in \cite{Brenner_Constantin_Leo_Schenkel_Venkataramani_steady_state_99, MR1651769,nguyen2025infinitely}.
In particular, there is an explicit self-similar solution given by
\be
\rho_{s*}(t,x) = \frac{1}{T-t}Q_{s*}\left( \frac{x}{\sqrt{T-t}} \right), \quad \text{ with } \quad Q_{s*}(x) = \frac{4(6+|x|^2)}{(2 + |x|^2)^2},
\label{profile}
\ee
whose radial stability has been verified by Glogi\'c and Sch\"orkhuber \cite{stable3DKS}. Recently, Li and the third author \cite{ksnsblowup,lizhou2025nonradial} extended this stability theory to the nonradial setting.
Beyond self-similar blowup solutions, other non-self-similar formations have also been identified. For example, Collot, Ghoul, Masmoudi, and Nguyen \cite{Collot_Ghoul_Masmoudi_Nguyen_3Dblowup_Collasping-ring_blowup23} discovered a collapsing-ring blowup solution, and
Nguyen, Nouaili, and Zaag \cite{nguyen2023construction} found a type I blowup solution with a log correction on the shrinking rate. Additionally, Hou, Nguyen, and Song \cite{hou2025axisymmetric} recently found a type II blowup solution under axisymmetry, whose local leading-order profile coincides with the rescaled stationary solution of the two-dimensional Keller-Segel equation as given in \eqref{type II blowup: profile}.

    \subsubsection{Blowup or no blowup with logistic damping}
    \label{subsub: logistic damping}

    \mbox{}

    \vspace{0.1cm}

    To detect the core mechanism behind the bacterial aggregation, we can rewrite the aggregation term in \eqref{equation, Keller-Segel} or \eqref{equation: KS with damping} as
\begin{equation*}
    \nabla \cdot (\rho \nabla \Dein \rho)  = \nabla \Dein \rho \cdot \nabla \rho + \rho^2.
\end{equation*}
In this expression, $\na \Dein \rho \cdot \na \rho$ is the advection term, which, along with the effect of the diffusion term, cannot lead to the finite-time blowup. Consequently, the main factor causing the blowup is $\rho^2$.

If the logistic damping term $-\mu \rho^2$ in \eqref{equation: KS with damping} is replaced by a stronger term of the form $-\mu \rho^p$ with $\mu>0$ and $p >2$, then the aggregation effect is always suppressed. Precisely, the smooth solution always exists globally in time, regardless of how concentrated the initial density is, see \cite{Tello_Winkler_globalexi_for_p=2}. 

In the scenario of quadratic damping
$-\mu \rho^2$ (corresponding to \eqref{equation: KS with damping}), Tello and Winkler \cite{Tello_Winkler_globalexi_for_p=2} proved that the global existence is guaranteed for any $\mu >0$ in two dimensions.
For higher dimension $d\geq 3$, they further demonstrated that global existence holds provided $\mu > \frac{d-2}{d}$. Subsequently, Kang and Stevens \cite{Kang_Stevens_globalexistence_p=2_and_mu_critical} verified the global existence in the critical case $\mu = \frac{d-2}{d}$ with $d \ge 3$. In the subcritical regime with $\mu \in \left(0, \frac{d-2}{d}\right)$, Fuest \cite{Fuest_blowup_optimal_2021} partially bridged the gap by proving that finite-time blowup occurs in a bounded domain for $\mu \in \left( 0, \frac{d-4}{d} \right)$ in higher dimensions $d \ge 5$. However, to the best of the authors' knowledge, 
prior to the present work, it remained an \textit{open problem} whether blowup occurs for $\mu \in (\frac{d-4}{d}, \frac{d-2}{d})$ with $d \ge 5$ and $\mu \in (0, \frac{d-2}{d})$ with $ 3 \le d \le 4$.

Regarding the weaker damping term $-\mu \rho^p$ with $1<p<2$, for higher dimensions $d \ge 5$, Winkler \cite{Winkler_blowup_N>=5_2011} verified the existence of finite-time blowup solution when $1 <p < \frac{3}{2} + \frac{1}{2d-2}$. Subsequently, for dimensions $d \in \{3,4\}$, Winkler \cite{Winkler_blowup_N=3_or_4_2018} demonstrated finite-time blowup for $1< p < \frac{7}{6}$. This range was later significantly extended by Fuest \cite{Fuest_blowup_optimal_2021}, where the author established the finite-time blowup for $1 < p< \frac{3}{2}$ in dimensions $d=3$ and for  $1 < p< 2$ in dimensions $d \ge 4$.

    \subsection{Main result}

    \mbox{}
    
    The principal objective of this work is to establish the \textit{optimal blowup result} to \eqref{equation: KS with damping} in three dimensions. Specifically, we establish the existence of a smooth finite-time blowup solution to \eqref{equation: KS with damping} with nonnegative density and finite mass for any $0 \le \mu < \frac{1}{3}$ in three dimensions.

    In the literature, most blowup results for chemotaxis equations are obtained either by tracking the evolution of some appropriate functional (such as the second moment) to obtain a contradiction  \cite{Blanchet_Dolbeault_Perthame_globalexistence06, MR2099126, Dolbeault_Perthame_globalexistence04}, or by working in the radial setting where the mass accumulation function satisfies an ODE  \cite{Fuest_blowup_optimal_2021}. However, the damping term $-\mu \rho^2$ in \eqref{equation: KS with damping} destroys the structures that these proofs rely on. In particular, it \textit{leads to a time-decreasing mass} and \textit{breaks the divergence form structure} that the classical Keller-Segel equation \eqref{equation, Keller-Segel} enjoys, thereby making these classical methods too limited to obtain the sharp blowup result for \eqref{equation: KS with damping}.

    Motivated by the various singularity formation results, 
     ranging from the nonlinear heat equation \cite{bricmont1994universality,MR3986939,hou20242,MR1427848}, nonlinear wave equation \cite{MR3537340,kim2022self}, nonliear Schr\"odinger equation \cite{MerleblowupNLSdefocusing,MR2729284}, incompressible fluids \cite{Chen_Hou_Euler_blowup_theory_2023, Chen_Hou_Euler_blowup_numerical_2023,Chen_Hou_Huang_De_Gregorio_eq_2021,hou2024blowup,MR4334974,MR4445341}, compressible fluids \cite{Javierblowup3DNS,chen2024vorticity,MR4445443},  we adopt a strategy based on the direct construction of finite-time blowup solutions via the stability analysis of an approximate blowup solution. And the precise statement of the main theorem is as follows: 
    
\begin{theorem}[Existence of finite-time blowup to \text{\eqref{equation: KS with damping}} with finite mass]
\label{thm: existence of blowup}
For any fixed $0 \le \mu < \frac{1}{3}$, let $j_0$ be an integer such that
\be 
j_0 \ge J:= \frac{3(1-\mu)}{1-3\mu} +1 >1,
\label{J: def}
\ee
it then follows that
\be
\beta = \beta(j_0,\mu) := \frac{1}{3(1-\mu)} + \frac{1}{2j_0} < \frac{1}{2}.
\label{beta: def, thm}
\ee
In addition, under these choices, there exists a radially nonnegative $\rho_0 \in C_0^\infty (\RR^3)$  \footnote{Since $\rho_0 \in C_0^\infty(\RR^3)$, the total mass of the system is necessarily finite. Additionally, with the nonnegativity of initial data $\rho_0$, an argument analogous to \cite[Theorem A.1]{ksnsblowup} ensures that the corresponding solution remains nonnegative in its lifespan.
}
such that the related smooth solution to \eqref{equation: KS with damping} 
exhibits finite-time blowup at $t=T<\infty$. Moreover,
the solution satisfies
\be
\rho(t,x) = \frac{1}{T-t} \left[ Q \left( \frac{x}{(T-t)^\beta} \right) + \ep \left( t, \frac{x}{(T-t)^\beta} \right) \right], \quad x \in \RR^3, \; t \in [0,T),
\label{blowup KS with damping}
\ee

\noindent where 

\noindent $\bullet$ (Existence of profile). $Q$ is a unique smooth
radial decreasing solution to the equation:
\be
    Q+\beta y \cdot \nabla_y Q = \nabla_y Q \cdot \nabla_y\Delta^{-1}_yQ + (1-\mu)Q^2,
    \label{PDE Q}
\ee
with the initial data
 \be
    Q(0) = \frac{1}{1-\mu} \quad \text{ and } \quad \partial_r^{(2j_0)} Q(0) = -(2j_0)!  <0.
    \label{smooth profile: initial data -1}
\ee

\noindent $\bullet$ (Smallness of error term).
There exists $s=s(\mu) \in \ZZ_{\ge 1}$ sufficiently large and constants $C>0$ and $\bar \epsilon >0$ such that the error term $\ep$ satisfies
\[
\| \ep(t) \|_{H^s(\RR^3)} \le C \left( T-t \right)^{\bar \epsilon}, \quad \forall \; 0 \le t <T.
\]
    
\end{theorem}

\vspace{0.15cm}

\noindent \textit{Comments on Theorem \ref{thm: existence of blowup}.}
\label{ref:com}

\vspace{0.08cm}

\noindent \textit{$1.$ Optimality in blowup results for Keller-Segel equation with logistic damping.}

\vspace{0.1cm}

Theorem \ref{thm: existence of blowup} shows the existence of a finite-time blowup solution to $3$D Keller-Segel equation with logistic damping \eqref{equation: KS with damping} for any fixed $0 \le \mu < \frac{1}{3}$. This result provides the optimal blowup threshold for \eqref{equation: KS with damping} and completes the long-standing conjecture proposed in \cite{Fuest_blowup_optimal_2021,Tello_Winkler_globalexi_for_p=2} as mentioned in Subsection \ref{subsub: logistic damping}.

Moreover, our method is robust enough to extend naturally to higher dimensions $d \ge 3$, resulting in an analogous sharp blowup regime. Specifically, for $\mu \in \Big[ 0, \frac{d-2}{d} \Big)$ with $d \ge 3$, there always exists a finite-time blowup solution to \eqref{equation: KS with damping}.

When the damping term is weakened to $-\mu \rho^p$ with $\mu>0$ and $1 \le p <2$, the problem becomes simpler, as this term can also be viewed primarily as a perturbation of the dominant aggregation effect. Concretely, one could use $\frac{1}{T-t}Q\big(\frac{x}{(T-t)^\beta}\big)$ with $Q$ constructed in Theorem \ref{thm: existence of blowup} with $\mu=0$ as an approximate solution. Alternatively, one may use the self-similar solution \eqref{profile} of the classical Keller-Segel equation \eqref{equation, Keller-Segel} as an approximate solution. Then, we could apply the stability theory developed in this work or \cite{stable3DKS,lizhou2025nonradial} to establish the existence of finite-time blowup for any $\mu > 0$ and $p \in [1,2)$ in three dimensions. Since this problem is easier to handle, we do not pursue this in detail and only study \eqref{equation: KS with damping} with essential difficulty.

\vspace{0.2cm}

\noindent \textit{$2.$ New blowup mechanism for classical Keller-Segel equation in $3$D.}

\vspace{0.1cm}

For $3$D classical Keller-Segel equation \eqref{equation, Keller-Segel}, by applying Theorem \ref{thm: existence of blowup} with $\mu=0$, we have constructed a countably infinite family of blowup solutions, whose asymptotic behavior is given by
\[
\rho(t,x) \approx \frac{1}{T-t} Q_{\beta_j} \left( \frac{x}{(T-t)^{\beta_j}} \right), \qquad \forall \; t \in (0,T),
\]
where $\beta_j = \frac{1}{3} + \frac{1}{2j}< \frac{1}{2}$ for all $ j \ge 4$ and each $Q_{\beta_j}$ solves \eqref{PDE Q} with $\beta =\beta_j$. This blowup is an "abnormal" Type I blowup, in the sense that it does not match natural scaling \eqref{scaling invariance} \footnote{
We take the ansatz $\rho(t,x) = \frac{1}{(T-t)^a} Q \left( \frac{x}{(T-t)^b} \right)$ and plug it into \eqref{equation: KS with damping}. Requiring that all terms in the resulting equation have the same strength, the exponents must satisfy $a=2b=1$. This choice precisely coincides with the scaling symmetry \eqref{scaling invariance} of the equation. We refer to blowup solutions satisfying this form as exhibiting natural self-similar blowup, in analogy with \eqref{equation: KS with damping}.

}. Consequently, as a by-product of Theorem \ref{thm: existence of blowup}, we establish a new family of blowup mechanisms for $3$D classical Keller-Segel equation \eqref{equation, Keller-Segel}.

\vspace{0.2cm}

\noindent \textit{$3.$ Extension of Theorem \ref{thm: existence of blowup}.}

\vspace{0.1cm}
Theorem \ref{thm: existence of blowup} establishes the precise blowup mechanism for the density but does not address its stability. Nevertheless, by following the standard approach outlined in \cite{MR3986939, lizhou2025nonradial}, one can construct a finite-codimensional Lipschitz stable manifold of radial initial data such that the corresponding solutions to \eqref{equation: KS with damping} exhibit blowup dynamics similar to in \eqref{blowup KS with damping}.

Besides, a natural extension to the nonradial blowup can be expected by using the spectral method from \cite{nonradialblowupNS, nonradialblowupNLS, chen2024vorticityhigherdimension, chen2024vorticity}. Alternatively, following the argument in \cite{chen2024stability}, one might introduce a matrix system of modulation ODEs to realize it. 
Furthermore, noting that the blowup solution constructed in Theorem \ref{thm: existence of blowup} is highly localized, one can expect a finite-time blowup solution in a bounded domain by adapting the cut-off technique from \cite{nonradialblowupNS, nonradialblowupNLS}. However, as the main focus of this work is on the existence of finite-time blowup solutions to \eqref{equation: KS with damping}, we do not pursue these generalizations in detail here.

In addition, our method is expected to be sufficiently robust to potentially shed light on other models, such as the nonlinear heat equations. As a further motivation for the linear and nonlinear analysis for the Keller-Segel equation \eqref{equation: KS with damping} in Section \ref{sec: linear theory} and Section \ref{sec: nonlinear sta}, we briefly outline this in Section \ref{sec:slh} for $1$D semilinear heat equation.

\vspace{0.23cm}

\subsection{Proof strategy and related key highlights.}

\vspace{-0.1cm}

\subsubsection{Key idea of the construction: diffusion term as a perturbation.}

\mbox{}

\vspace{0.08cm}

Motivated by the explicit self-similar blowup solution \eqref{profile} to the classical Keller-Segel equation \eqref{equation, Keller-Segel}, a direct idea is to seek an analogous self-similar solution matching the natural scaling for the damped system \eqref{equation: KS with damping}. Specifically, since \eqref{equation: KS with damping} satisfies the scaling invariance \eqref{scaling invariance}, one is led to consider special solutions of form
\vspace{-0.05cm}
\[
\rho_\mu(t,x) = \frac{1}{T-t} Q_{\mu} \left( \frac{x}{\sqrt{T-t}} \right),
\]
where $Q_\mu$ solves
\be
Q_\mu + \frac{1}{2} y \cdot \na Q_\mu = \Delta Q_\mu + \na \cdot (Q_\mu \na \Dein Q_\mu) - \mu Q_\mu^2.
\label{self-simiar profile: equation}
\ee
However, the damping term $-\mu \rho^2$ significantly complicates the problem by breaking the divergence structure of \eqref{equation, Keller-Segel}. In particular, the partial mass 
\be
m_\rho (r) := \int_{B(0,r)} \rho(x) dx,
\label{partial mass: def}
\ee
is invalid to simplify the equation \eqref{self-simiar profile: equation} into a second order local differential equation, resulting \eqref{self-simiar profile: equation} essentially a third-order differential equation for general $\mu \in \left( 0, \frac{1}{3} \right)$, which, to the best of authors' knowledge, is quite difficult to analyze and is still open.

Instead, we adopt an alternative approach inspired by \cite{chen2024stability,hou20242,MR4445442,MR4445443,nguyen2023construction,MR4752990}, 
where the first step is to neglect the diffusion term
and construct a self-similar blowup solution for the following $3$D aggregation equation
\be
    \pa_t \rho = \nabla \rho \cdot \nabla\Delta^{-1}\rho + (1-\mu)\rho^2,
    \label{aggregation eq}
\ee
in the form
\be
    \rho(t,x) = \frac{1}{T-t} Q\left( \frac{x}{(T-t)^\beta} \right), \;  \text{ with } \beta < \frac{1}{2},
    \label{abnormal blowup}
\ee
where the profile $Q$ satisfies \eqref{PDE Q}. We then choose \eqref{abnormal blowup} as an approximate solution to the original system \eqref{equation: KS with damping}, which breaks the natural scaling of \eqref{equation: KS with damping}, thus making the diffusion term
$\Delta \rho$ subcritical and allowing it to be treated as a perturbation relative to the dominant aggregation term. 

Compared to \eqref{self-simiar profile: equation}, the equation \eqref{PDE Q} is effectively second-order and much more tractable. 
By restricting to radial solutions and introducing the averaged mass of $Q$ in the ball $B(0,r)$ 
\be
    f_Q(r) := \frac{1}{r^3}\int_0^r Q(s)s^2 ds, 
    \label{f: definition}
\ee  
\eqref{PDE Q} can be reformulated as an ODE system in terms of $(Q,f_Q)$ (see \eqref{ODE system}).
The formation of solutions to \eqref{ODE system} is through the dynamical system and a standard phase portrait analysis (See Figure \ref{fig3} and Lemma \ref{lem: solve ODE}). These methods have been instrumental in recent developments concerning the construction of smooth profiles for implosion, gravitational collapse, and collapsing-expanding shocks \cite{Javierblowup3DNS,ghj21,guoMahirJangpolytropicgap,jls2025,jang2023selfsimilarconvergingshockwaves,MR4445443}. The regularity of self-similar solutions serves as a fundamental criterion for solution admissibility and plays a crucial role in ensuring the existence of a nontrivial solution in the context of the present problem.

\subsubsection{Compatibility of $\mu < \frac{1}{3}$: insights from the phase portrait analysis.}

\mbox{}

\vspace{0.1cm}

One of the main goals of our paper is to figure out the existence of the smooth nontrivial radial solution to \eqref{PDE Q} for some $\beta < \frac{1}{2}$, which can be further simplified into finding a smooth curve connecting \be
P_0 := \left( \frac{1}{1-\mu}, \frac{1}{3(1-\mu)} \right), \footnote{
With the singular term $\frac1r$ appearing on the right-hand side of the ODE system \eqref{ODE system},
the choice of initial data $(Q(0),f_Q(0))=P_0$ should be required to ensure that $Q$ is smooth near $r =0$.
}
\label{P_0: def}
\ee
and the origin $O$ under the ODE system for $(Q,f_Q)$ given by \eqref{ODE system}. We classify the problem into three cases based on the relative position between $f_Q = \beta$ and $P_0$, and analyze the corresponding phase portraits:

\vspace{0.1cm}

\noindent \textit{\textbf{Case I.} $P_0$ lies strictly above the line $\{f_Q=\beta\}$: $\beta < \frac{1}{3(1-\mu)}$ (See Figure \ref{fig1}).}
\vspace{-0.2cm}
\begin{figure}[h]
\centering
\captionsetup{width=.9\linewidth}
\includegraphics[width=0.69\textwidth]{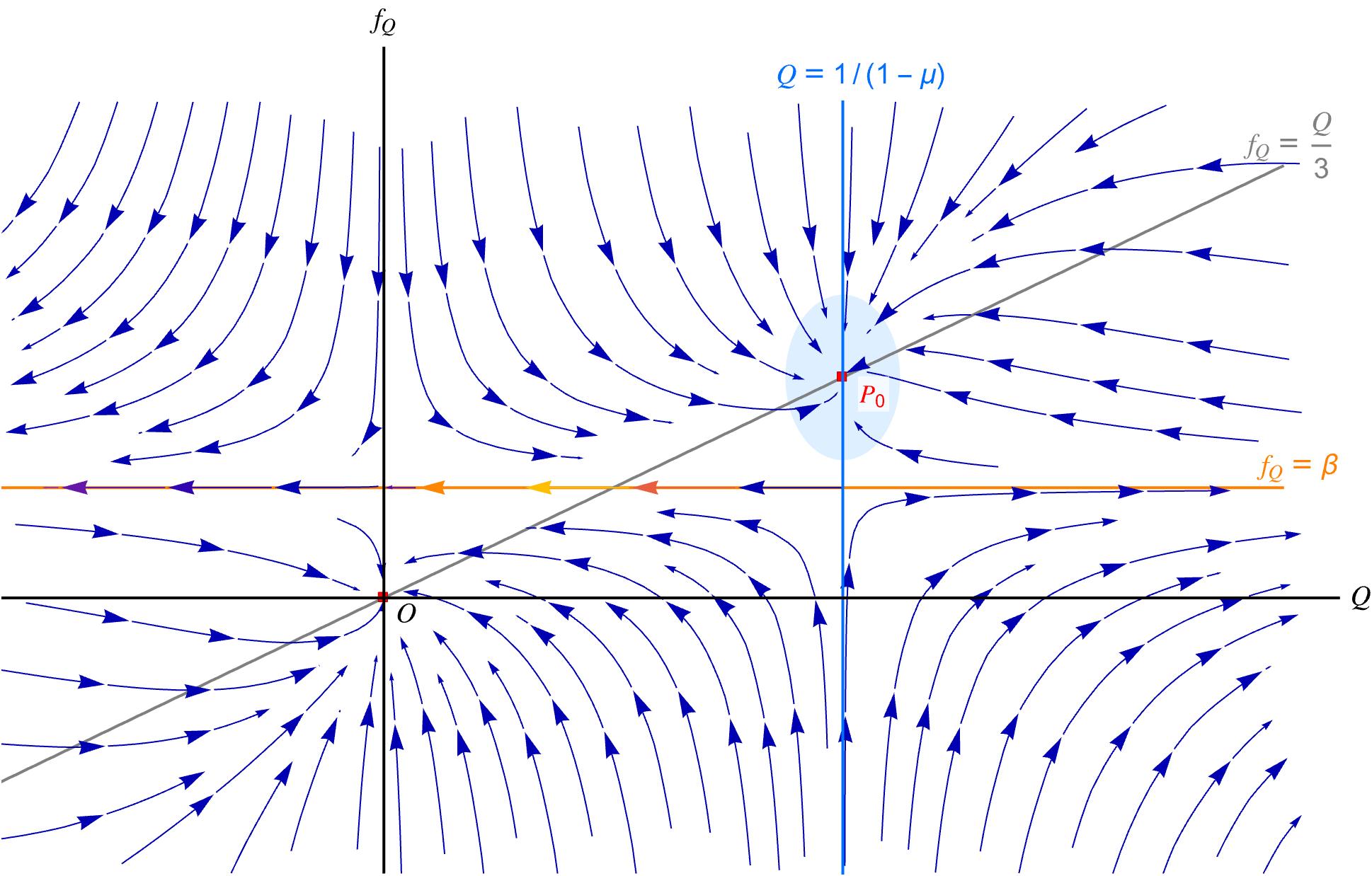}
\caption{$Qf_Q$-plane with $P_0$ above the line $f_Q=\beta$. In this case, there is only a trivial solution solving \eqref{ODE system} with $(Q(0),f_Q(0))=P_0$.
}
\label{fig1}
\vspace{-2mm}
\end{figure}

\vspace{0.05cm}

We remark here that in \textit{Case I} $\big(\beta < \frac{1}{3(1-\mu)} \big)$, motivated by the related phase portrait (see Figure \ref{fig1}), the flow field around $P_0$ is directed toward $P_0$. Consequently, $P_0$ is a stagnant point in this case and $Q(r) \equiv \frac{1}{1-\mu}$ is the only smooth solution to \eqref{ODE system} with $(Q(0),f_Q(0))=P_0$ that can be obtained. 
Therefore, to obtain a nontrivial solution to \eqref{ODE system}
with $(Q(0),f_Q(0))=P_0$, it is necessary  
for $\beta$ to satisfy
\be
\frac{1}{3(1-\mu)} \le \beta < \frac{1}{2}.
\label{beta: range of choice}
\ee
This requires that $\mu< \frac{1}{3}$, which aligns exactly with the \textit{optimal range of $\mu$} provided in Subsection \ref{subsub: logistic damping} and Theorem \ref{thm: existence of blowup}.

\vspace{0.1cm}

\noindent \textit{\textbf{Case II.} $P_0$ lies on the line $\{f_Q=\beta\}$: $\beta = \frac{1}{3(1-\mu)}$ (See Figure \ref{fig2}).}

\begin{figure}[h]
\centering
\captionsetup{width=.9\linewidth}
\includegraphics[width=0.69\textwidth]{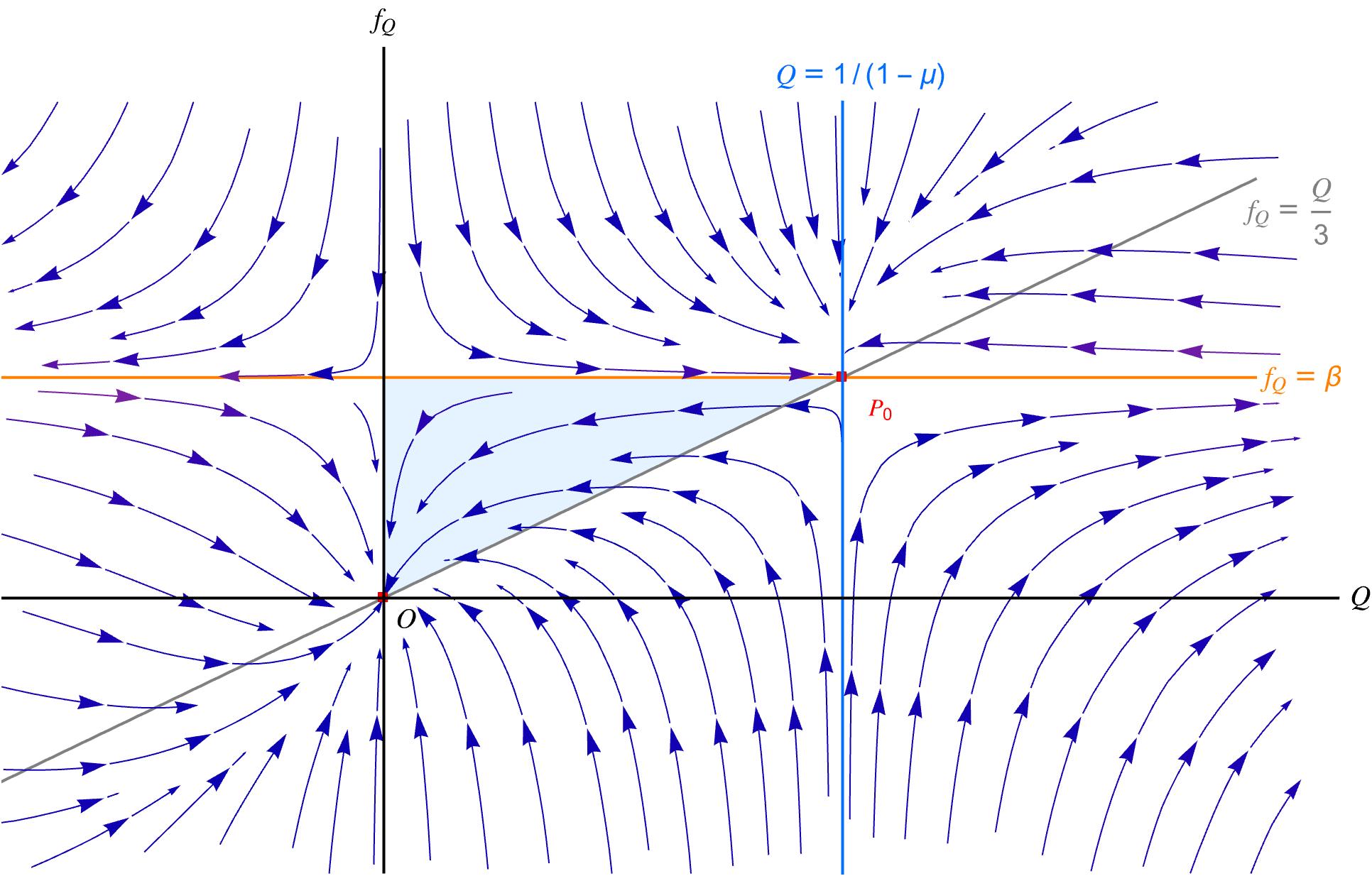}
\caption{$Qf_Q$-plane with $P_0$ on the line $f_Q=\beta$. In this case, though it seems possible to have a solution curve connecting between $P_0$ and the origin $O$ from the phase portrait, with the overwhelming singularity on the right-hand side of \eqref{ODE system}, after some careful analysis, the only possible smooth solution starting from $P_0$ to should be the trivial one: $(Q(r),f_Q(r)) \equiv P_0$.}
\label{fig2}
\end{figure}

\vspace{0.1cm}

\noindent \textit{\textbf{Case III.} $P_0$ lies strictly below the line $\{f_Q=\beta\}$: $\beta > \frac{1}{3(1-\mu)}$ (See Figure \ref{fig3}).}

\begin{figure}[h]
\centering
\captionsetup{width=.9\linewidth}
\includegraphics[width=0.69\textwidth]{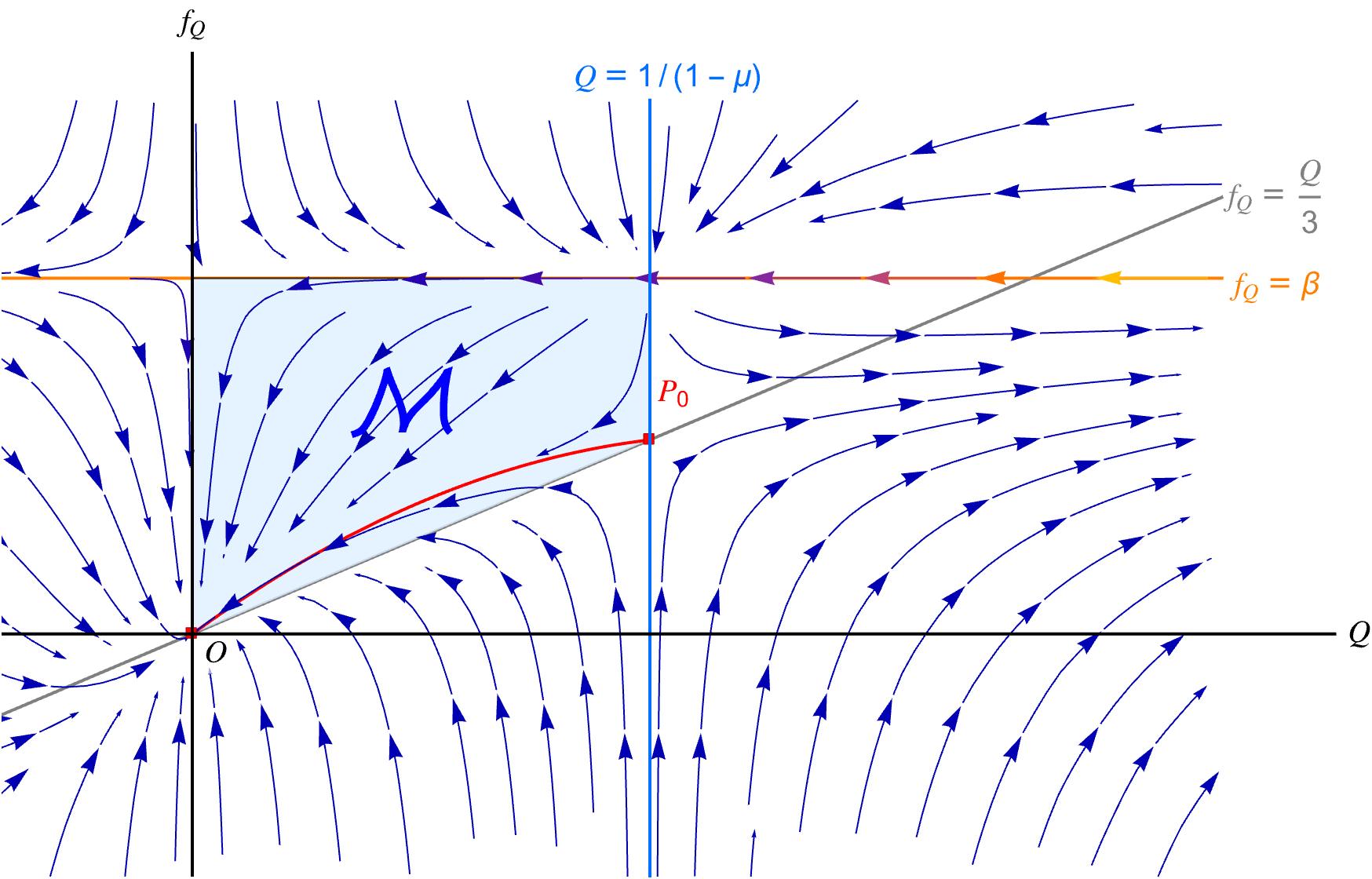}
\caption{$Qf_Q$-plane with $P_0$ below $f_Q = \beta$. In this case, we can always choose countably many $\{ \beta_j \}_j$  (see Theorem \ref{thm: existence of blowup} and Lemma \ref{lem: solve ODE} for more details) such that there exists a smooth solution curve (the red curve in the figure) lying in $\calM$ and connecting $P_0$ and the origin $O$ simultaneously, we will discuss this case in more details later in Lemma \ref{lem: solve ODE}.
}
\label{fig3}
\end{figure}

\subsubsection{Stability analysis beyond the spectral theory.}

\mbox{}

\vspace{0.1cm}

After selecting the self-similar blowup solution \eqref{abnormal blowup} of the aggregation equation \eqref{aggregation eq} as an approximate solution to \eqref{equation: KS with damping}, we then proceed to study the perturbation dynamics near the profile $Q$ under the self-similar coordinate \eqref{ss coordinte}, which constitutes the key to our stability analysis. 

In the literature, one of the natural approaches to address this problem is to analyze the spectral properties of the associated linearized operator. Such analysis typically relies on compactness arguments \cite{Javierblowup3DNS, nonradialblowupNLS, engel2000one, jiaSverakillposedness, ksnsblowup,  MerleblowupNLSdefocusing, MR4445443} or other more intricate operator theory, e.g. studying the precise or approximated spectrum of the linearized operator \cite{Collot_Ghoul_Masmoudi_Nguyen_3Dblowup_Collasping-ring_blowup23, hou2025axisymmetric, lizhou2025nonradial, MR1427848, nguyen2023construction}.  For problems with an explicit profile available, studying the spectrum of the linearized operator around the profile yields useful information for stability, e.g. \cite{Collot_Ghoul_Masmoudi_Nguyen_3Dblowup_Collasping-ring_blowup23,lizhou2025nonradial, MR1427848}.
However, for more complicated dynamics with only implicit or numerically approximate profiles, as is our case, obtaining the spectral information becomes significantly more challenging.

Alternatively, we adopt a robust approach based on energy estimates with a singular weight appropriately chosen. Linear stability can be extracted via enforcing local vanishing conditions of the perturbation at the origin, using only limited information of the profile. The idea was first demonstrated via the seminal works of \cite{Chen_Hou_Euler_blowup_theory_2023,Chen_Hou_Huang_De_Gregorio_eq_2021} in the self-similar case, and later on generalized by the second author with collaborators to type I singularities beyond self-similarity \cite{chen2024stability,hou20242} with the blowup law automatically inferred. Such a stability argument via singular weights will be further amenable to computer-assisted proofs. In this paper, we further generalize the idea to the finite-codimensional stability case by studying this open problem.

\subsection{Structure of the paper.}

\mbox{}

\vspace{0.1cm}
In Section \ref{sec:slh}, we briefly discuss the approach for handling the finite-codimensional stability of blowup solutions to the $1$D semilinear heat equation by introducing the singular weights. Motivated by this illustrative example, 
we then focus on the Keller-Segel equation with logistic damping term \eqref{equation: KS with damping}. In Section \ref{sec: ODE}, given any fixed $\mu \in \left[ 0, \frac{1}{3} \right)$, we establish the existence of nontrivial solutions to equation \eqref{PDE Q} with an appropriate choice of $\beta \in \left( \frac{1}{3(1-\mu)},\frac{1}{2} \right)$, the quantitative properties of the approximate profiles will be investigated as well. 
For the stability analysis, we first examine the coercivity of the linearized operator in the $L_w^2$ sense with $w$ being sufficiently singular at the origin, as detailed in Section \ref{sec: linear theory}. Finally, in Section \ref{sec: nonlinear sta}, we will study the nonlinear stability to complete the proof of Theorem \ref{thm: existence of blowup}.

\subsection{Notations.} 
We denote 
\[
\la r \ra = \sqrt{1+r^2}.
\]
And we define a radial smooth cut-off function $\chi$ on $B(0,1)$ by
\be
\chi(x) 
= \begin{cases}
    1, & \text{if } |y| \le 1, \\
    0, & \text{if } |y| \ge 2.
\end{cases}
\label{def: cut-off function}
\ee
Based on this, we define the smooth cut-off function on $B(0,R)$ by
\be
\chi_R(x) : = \chi \left( \frac{x}{R} \right).
\label{def: cut-off on B(0,R)}
\ee

For any given weighted function $w:\RR^3 \to \RR$, we define $L_w^2(\RR^3)$ the weighted inner product by
\[
(g_1,g_2)_{L_w^2(\RR^3)} = \int_{\RR^3} g_1(x) g_2(x) w(x) dx,
\]
and the weighted $L^2$ space $L_w^2(\RR^3)$ by the collections of all functions of $g$ satisfying $\| g \|_{L_w^2}^2:= (g,g)_{L_w^2} <\infty$.

Furthermore, we define $C_0^\infty(\Omega)$ as the collection of all $C^\infty(\RR^3)$ functions with compact support in $\Omega \subset \RR^3$.

For any fixed $m \in \ZZ_{\ge 0}$, we define $\dot H^{2m}$ the collection of all functions satisfying
\[
\| g \|_{\dot H^{2m}}^2 = \int_{\RR^3} |\Delta^m g|^2 dx < \infty,
\]
and define $\dot H^{2m+1}$ the collection of all functions satisfying
\[
\| g \|_{\dot H^{2m+1}}^2 = \int_{\RR^3} |\na \Delta^m g|^2 dx < \infty.
\]
In addition, we denote $H^k(\RR^3)$ for the collection of all functions with finite $\dot H^m$ norm for all $0\le m\le k$. And $H_{rad}^{k}$ is the collection of all radial functions lying in $H^{k}$. Moreover, wewrite $H^\infty(\RR^3) := \cap_{k=0}^\infty H^k(\RR^3)$.

If $h(x)=h \left( |x| \right)$ is a smooth radial function on $\RR^3$, then $h(r)$ can be approximated by the Taylor expansion at the origin: 
\be
h(r) = \sum_{j=0}^M [h]_j r^{2j} + O(r^{2M+2}), \; \text{ with } \; [h]_j := \frac{\partial_r^{(2j)} h}{(2j)!} \Bigg|_{r=0},
\label{Taylor coefficient: def}
\ee
where $[h]_j$ is the $2j$-th order coefficient of Taylor expansion of $h(r)$ at the origin.

\mbox{}

\begin{acknowledgement}
 Jiaqi Liu was supported by NSF grant DMS-2306910, and Yixuan Wang was supported by NSF Grant DMS-2205590. The authors gratefully acknowledge Yao Yao for valuable discussions and insightful suggestions throughout this research. We would like to express our great gratitude to Zhongtian Hu for the informative discussions. 
 Our thanks also extend to Thomas Y. Hou, Juhi Jang, Van Tien Nguyen, Xiang Qin, Jia Shi, and Peicong Song for their valuable suggestions to our work. Finally, we thank the IMS workshop on "Singularities in Fluids and General Relativity" for providing the opportunity for collaboration.

\end{acknowledgement}

\section{Motivating example of 1D semilinear heat equation}\label{sec:slh}

In this section, to illustrate the ideas of proof more clearly, we first consider the simpler $1$D semilinear heat equation as a motivating example.
Precisely, we will briefly sketch the high-level idea to study high-order vanishing type-I blowup for the $1$D semilinear heat equation under \textit{radial (even symmetric)} setting:
\be
u_t=u_{xx} + u^2.
\label{PDE: nonlinear heat}
\tag{HEAT}
\ee

\subsection{Self-similar renormalization and the approximate solution} For any fixed $m \in \ZZ_{> 1}$, we introduce the self-similar coordinate
\be
y = \frac{x}{\lambda^{\frac{1}{m}}}, \quad \frac{d\tau}{dt} = \frac{1}{\lambda^2}, \quad \tau\Big|_{t=0} =0, \quad 
\frac{\lambda_\tau}{\lambda} = -\frac{1}{2},
\label{ss coordinate: heat eq}
\ee
and corresponding renormalization
\be
u(t,x) = \frac{1}{\lambda^2} U \left( \tau, y \right),
\label{ss renormalization: heat eq}
\ee
then $\lambda(\tau) = \lambda_0 e^{-\frac{1}{2} \tau}$ and $U$ solves the equation
\be
\partial_\tau U = \lambda^{2-\frac{2}{m}} U_{yy} - U - \frac{1}{2m} y \cdot \na U + U^2,
\label{eq: heat renormalized setting}
\ee
where the diffusion term can be regarded as a perturbation since $2- \frac{2}{m} >0$. This, in turn, motivates us to find an approximate solution to the equation
\be
-U - \frac{1}{2m} y \cdot \na U + U^2=0.
\label{sec: PDE U*}
\ee
In particular, this equation can be explicitly solved by
\be
U_* (y)= (1+ c y^{2m})^{-1}.
\label{sec3: profile}
\ee
Here $c>0$ is a constant, and $m>1$ describes the vanishing order of the next order expansion of $U_*$ near the origin.

\subsection{Linear stability}
We fix $m>1$ and $c=1$ in \eqref{sec3: profile}, and plug the ansatz  $U = U_* + \ep$ into \eqref{eq: heat renormalized setting}, it then follows that $\ep$ solves
\[
\partial_\tau \ep = \lambda^{2-\frac{2}{m}} U_{yy} + \calL \ep + \ep^2,
\]
where the linearized operator reads
\be
\mathcal{L}\varepsilon=-\varepsilon-\frac{1}{2m}y \partial_y \varepsilon +2U_*\varepsilon.
\label{sec: L}
\ee
Next, we introduce a weighted $L_\Theta^2$ space with singular weight $\Theta(y) = y^{-4m-4}$ near the origin to extract damping, which is the essential step to close the nonlinear stability. Via the integration by parts, we have the coercivity near the origin
\be
\left( \calL \ep, \ep \right)_{L_\Theta^2}
= \left( \left( -1 + 2U_* + \frac{(\Theta y)_y}{4 m \Theta} \right) \ep, \ep \right)_{L_{\Theta}^2} 
\approx  - \frac{3}{4m} \left(\ep, \ep \right)_{L_{\Theta}^2}.
\label{sec3: cor L 1}
\ee
In particular, with careful analysis, there is a small constant $0< \kappa = \kappa(m,U_*) \ll 1$ such that \eqref{sec3: cor L 1} can be extended to
\be
\left( \calL \ep, \ep \right)_{L_{\Theta+\kappa}^2}
 \le - \frac{1}{4m} \left(\ep, \ep \right)_{L_{\Theta+\kappa}^2}.
\label{sec3: cor L 2}
\ee

Additionally, we need to introduce the higher Sobolev norm $\dot H^{\bar K}$ to close the bootstrap argument. Precisely,
\begin{align}
\left( \calL \ep, \ep \right)_{\dot H^{\bar K}}
& =\left( \left(-1-\frac{\bar K}{2m}+2U_*+\frac{1}{4m} \right)\partial_y^{\bar K}\varepsilon, \partial_y^{\bar K}\varepsilon \right)_{L^2} + O(\| \ep \|_{H^{\bar K-1}} \| \ep \|_{\dot H^{\bar K}}) \notag \\
& \le - \frac{2\bar K-4m-1}{4m} \| \ep \|_{\dot H^{\bar K}}^2 + O(\| \ep \|_{H^{{\bar K}-1}} \| \ep \|_{\dot H^{\bar K}}),
\label{sec3: cor L 3}
\end{align}
where the leading order enjoys damping once we choose $\bar K = \bar K(m) \gg 1$.

\subsection{Modulation ODEs and nonlinear stability} With the singular weight $\Theta(y) = |y|^{-4m-4}$ given previously, we further \textit{radially} decompose $\ep$ into
\be
\ep (\tau,y)= \epu(\tau,y) + \eps(\tau,y), \quad \text{ with }  \quad \epu = \sum_{j=0}^{m} c_j(\tau) \chi(y) y^{2j},
\ee
such that $\eps(\tau,y) = O(y^{2m+2})$ near the origin, which yields an ODE system for modulation parameters $\{ c_j \}_{j=0}^m$:
\be
\begin{cases}
    \dot c_j = \left( 1- \frac{j}{m} \right) c_j + [\epu^2]_j + \lambda^{2- \frac{2}{m}} [U_{yy}]_j, & 0 \le j < m-1,
    \vspace{0.1cm} \\
    \dot c_{m} = [\epu^2]_m - 2c_0 + \lambda^{2- \frac{2}{m}}[U_{yy}]_m, & j=m.
\end{cases}
\ee
Additionally, $\eps = O(y^{2m+2})$ solves the equation
\[
\partial_\tau \eps =   \calL \eps  + 2 \epu \eps + \eps^2 + G[\lambda,U,\epu],
\]
with the modulation term $G[\lambda,U,\epu]= O(y^{2m+2})$ given by 
\[
G[\lambda,U,\epu]= \left( \lambda^{2-\frac{2}{m}} U_{yy} + \calL \epu + \epu^2 \right) - \sum_{j=0}^K \left[\lambda^{2-\frac{2}{m}} U_{yy} + \calL \epu + \epu^2 \right]_j \chi  y^{2j}.
\]
Finally, we can use the standard topological argument together with \eqref{sec3: cor L 1} and \eqref{sec3: cor L 3} to derive the nonlinear stability with finite codimension $m+1$. 

\begin{remark}
    We expect that this nonlinear stability result can be improved to finite-codimension $m-1$, which is two dimensions lower than our previous findings. 
    The key underlying reason is the presence of two degrees of freedom, namely, the choice of the blowup time $T>0$ and the shrinking rate. These degrees of freedom can be utilized through a matching argument to recover the corresponding unstable directions as in \cite{Limodestability, lizhou2025nonradial}.

    Alternatively, one may employ a method of dynamical rescaling to establish stability with finite codimension $m-1$. 
    Specifically, we modify the coordinate \eqref{ss coordinate: heat eq} as
    \[
    y = \frac{x}{\mu^\frac{1}{m}}, \quad  \frac{d\tau}{dt} = \frac{1}{\lambda^2}, \quad \tau\big|_{t=0}=0, \quad \frac{\lambda_\tau}{\lambda} = -\frac{1}{2} + \frac{1}{2} c_a, \quad \frac{\mu_\tau}{ \mu} = - \frac{1}{2}  - mc_s.
    \]
    We then define the corresponding renormalization
    \[
    u(t,x) = \frac{1}{\lambda^2} U (\tau,y).
    \]
    Under this coordinate transformation, $U$ satisfies
    \[
    \partial_\tau U = \lambda^2 \mu^{-\frac{2}{m}} U_{yy} + \left( -1 + c_a \right) U -\left( \frac{1}{2m} + c_s \right) y \partial_y U + U^2,
    \]
    where the parameters $(c_a,c_s)$ are determined by the modulation conditions
    \[
    U(\tau, 0) = U_*(0) \qquad \text{ and } \qquad  [U(\tau,0)]_{m}= [U_*]_{m},
    \] 
    and we eliminate neutral modes by fixing $c_0=c_m=0$.
    By applying a similar argument of modulation ODEs, we obtain the nonlinear stability with finite codimension $m-1$. Notably, introducing extra scaling parameters to perturb the scaling symmetry is crucial for extending the argument to the nonradial setting; see previous works by the second author and collaborators \cite{chen2024stability,hou20242}.
\end{remark}

\begin{remark}

Compared with the semilinear heat equation \eqref{PDE: nonlinear heat}, analyzing the Keller-Segel equation with logistic damping \eqref{equation: KS with damping} involves several additional challenges. For example, 
the profile $U_*$ 
introduced in \eqref{sec3: profile} is an explicit solution to 
the first-order and separable local equation \eqref{sec: PDE U*}. 
In contrast, for \eqref{equation: KS with damping} with $\mu\in \left( 0, \frac{1}{3} \right)$, the associated profile equation \eqref{PDE Q} is inherently nonlocal and cannot be trivially solved. 
This nonlocality requires a more delicate analysis, which will be carried out in Section \ref{sec: ODE}.

    Moreover, since there is no explicit nontrivial solution to \eqref{PDE Q} with \eqref{smooth profile: initial data -1}, additional effort is required to derive quantitative properties of the profile $Q$. Combined with the nonlocal nature of \eqref{equation: KS with damping}, these complexities make the establishment of linear coercivity of the Keller-Segel equation more intricate than in the case of the semilinear heat equation \eqref{PDE: nonlinear heat}.  Detailed strategies to handle these obstacles will be presented in Section~\ref{sec: linear theory}.

\end{remark}

    \section{Existence of profile via phase-portrait method}
    \label{sec: ODE}
This section is devoted to the existence of smooth profile solving \eqref{PDE Q} when $\mu< \frac{1}{3}$ with the appropriate choices of $\beta$. Firstly, note that for any radially symmetric function $R(y)$ satisfying $R(y) \to 0$ as $|y| \to \infty$,
\be
    \nabla\Delta^{-1}R (y) = \frac{y}{|y|^3} \int_0^{|y|} R(s)s^2 ds,
    \label{nonlocal term: formula}
\ee
Then, under the radial symmetric assumption, \eqref{PDE Q} becomes 
\be\label{ode Q}
    \left(\frac{1}{r^2}\int_0^r Q(s)s^2 ds-\beta r \right) Q'(r) -Q + (1-\mu)Q^2=0.
\ee
If $f_Q \not = \beta$, by introducing $f_Q$ (cf. \eqref{f: definition}), then we obtain an ODE system of $(Q(r),f(r))$ for $r >0$:
\be
\begin{cases}
   Q'(r) = \frac{1}{\beta -f_Q} \frac{(1-\mu) Q^2 - Q}{r},\vspace{0.2cm} \\
   f_Q'(r) = \frac{Q-3f_Q}{r}.
\end{cases}
\label{ODE system}
\ee

\begin{remark}
    For notational simplicity, in the subsequent discussion of this section, we introduce the following notations: we write $f :=f_Q$, and $Q_i := [Q]_i$ and $f_i := [f]_i$ for any $i \in \ZZ_{\ge 0}$, where $f_Q$ is given in \eqref{ODE system}, $[Q]_i$ and $[f]_i$ are the $2i$-th order coefficients of Taylor expansion of $f$ and $Q$ at the origin respectively (cf. \eqref{Taylor coefficient: def}).
\end{remark}

Our main result of this section is as follows, which gives the existence of the nontrivial smooth solutions to \eqref{ODE system} and the asymptotic behavior of the related solutions when $r=0$ and $r \to \infty$.

\begin{lemma}
\label{lem: solve ODE}
    For any fix $\mu\in[0,\frac13)$, let $\beta:=\beta(j_0,\mu) = \frac{1}{3(1-\mu)}+\frac{1}{2j_0}$ with arbitrarily fixed $j_0\geq J>1$ where $J$ is defined in \eqref{J: def}. Then, for any $Q_{j_0}<0$, there exists a unique smooth solution $(Q(r),f(r))$ to \eqref{ODE system} on $r \ge 0$ with initial conditions:
    \begin{align}
    (Q(0),f(0)) = \left( \frac{1}{1-\mu}, \frac{1}{3(1-\mu)} \right) \quad \text{ and } \quad 
    \partial_r^{2j_0} Q(0) = (2j_0)! Q_{j_0}.
    \label{ODE: initial assumption}
    \end{align}
  
    Moreover, the solution  $(Q(r),f(r))$ satisfies
    \begin{enumerate}
    \item  $(Q(r),f(r))$ is analytic near the origin. Precisely, there exists $\epsilon>0$ such that $(Q(r),f(r))$ can be expressed as Taylor series
\be
	Q(r) =\sum_{j =0}^{\infty} Q_{j}r^{2 j} \quad \text{and}\quad f(r) = \sum_{j =0}^{\infty} f_{j}r^{2 j}, \quad \forall \; r \in [0,\epsilon],
    \label{(Q,f), expansion}
\ee
where $\{Q_j\}_{j\geq 0}$ and $\{ f_j \}_{j\geq 0}$ satisfy the following recurrence relations
\be
\quad f_j = \frac{1}{2j+3} Q_j, \quad  \forall \; j \ge 0,
\label{recurrence relation 1}
\ee
and
\be
Q_j
=
\begin{cases}
\frac{1}{1-\mu}, & j=0, \vspace{0.1cm} \\
0, & 0< j < j_0, \vspace{0.1cm} \\
\frac{ \sum_{i=1}^{j-1} \left( \frac{2i}{2(j-i)+3} + (1-\mu)  \right)Q_i Q_{j-i} }{2j \left( \beta -\frac{1}{2j} - \frac{1}{3(1-\mu)} \right)}, & j >j_0.
\end{cases}
\label{recurrence Qj}
\ee
\item Both $Q(r),f(r)>0$ are strictly positive for $r\in[0,\infty)$ and monotonically decrease from $Q_0=\frac{1}{1-\mu}$ \big(respectively, $f_0 = \frac{1}{3(1-\mu)}$\big) to 0 as $r \to +\infty$. 
\item For any $j \ge 0$, there exists a constant $C=C(\beta,Q,j)>0$ such that
\be\label{decay estimates of Q}
|\partial_r^{(j)}Q(r)| + |\partial_r^{(j)}f(r)|\le C \la r \ra ^{-2-j} \quad \text{for} \quad  r \in \overline{\RR^+}.
\ee
In particular, $Q \in H^\infty (\RR^3)$.
    \end{enumerate}
\end{lemma}
\begin{proof}
\textit{Step 1. Solve \eqref{ODE system} via Taylor expansion.}
We assume that $(f(r),Q(r))$ can be expanded into the following forms:
\[
f(r) = \sum_{j=0}^\infty f_j r^{2j}, \quad Q(r) = \sum_{j=0}^\infty Q_j r^{2j},
\]
which yields that
\[
r f'(r)= \sum_{j=0}^\infty 2j f_j r^{2j}, \quad rQ'(r) = \sum_{j=0}^\infty 2j Q_j r^{2j}.
\]
Next, we substitute these expressions into the ODE system \eqref{ODE system} and match coefficients to establish \eqref{recurrence relation 1} and \eqref{recurrence Qj}. Precisely, we first expand the second equation in \eqref{ODE system} to get that
\[
\sum_{j=0}^\infty 2j f_j r^{2j} = \sum_{j=0}^\infty \left( Q_j -3 f_j \right) r^{2j}.
\]
which directly induces \eqref{recurrence relation 1}. 

Next, we expand $(\beta-f) r \partial_r Q = (1-\mu) Q^2- Q$, the first equation in \eqref{ODE system}. Using the identity
\[
\left( \sum_{j=0}^\infty h_j r^{2j} \right) \left( \sum_{j=0}^\infty k_j r^{2j} \right) = \sum_{j=0}^\infty \left( \sum_{i=0}^j k_i h_{j-i} \right) r^{2j},
\]
the left-hand side of the equation becomes
\begin{align*}
    \left( \beta -  \sum_{j=0}^\infty f_j r^{2j} \right) \left( \sum_{j=0}^\infty 2j Q_j r^{2j} \right) 
    & = \beta \sum_{j=0}^\infty 2j Q_j r^{2j} - \sum_{j=0}^\infty \left( \sum_{i=0}^j 2iQ_i f_{j-i} \right) r^{2j} \\ &= \sum_{j=0}^\infty \left(2 \beta j - \sum_{i=0}^j 2i Q_i f_{j-i} \right) r^{2j},
\end{align*}
and the right-hand side of the equation becomes
\begin{align*}
(1-\mu)\sum_{j=0}^\infty \left( \sum_{i=0}^j Q_i Q_{j-i} \right) r^{2j} - \sum_{j=0}^\infty Q_j r^{2j}  
 = \sum_{j=0}^\infty \left( (1-\mu) \sum_{i=0}^j Q_i Q_{j-i} -Q_j \right) r^{2j}.
\end{align*}
Comparing the coefficients on both sides, we obtain that
\[
2 \beta j Q_j - \sum_{i=0}^j 2i Q_i f_{j-i} = (1-\mu) \sum_{i=0}^j Q_i Q_{j-i} -Q_j, \quad \forall \;  j \ge 0.
\]
Replacing $f_{j-i}$ by using \eqref{recurrence relation 1}, we decouple the recurrence relations for $\{Q_j\}_{j=0}^\infty$ into
\[
2 \beta j Q_j - \sum_{i=0}^j \frac{2i}{2(j-i)+3} Q_i Q_{j-i} = (1-\mu) \sum_{i=0}^j Q_i Q_{j-i} -Q_j, \quad \forall \;  j \ge 0.
\]
In particular,
\[
(1-\mu) Q_0^2 -Q_0 =0, \quad \text{for } \; j=0,
\]
and $Q(0) = Q_0 = \frac{1}{1-\mu}$ exactly solves this equation. For $j\geq 1$, we isolate the highest index term $Q_j$ to obtain
\[
2j \Big(  \beta  - \frac{1}{2j} - \frac{1}{3(1-\mu)} \Big) Q_j = \sum_{i=0}^{j-1} \left( \frac{2i}{2(j-i)+3} + (1-\mu)  \right)Q_i Q_{j-i}.
\]
In particular, applying $j=1$ to the equation together with the fact that $\beta< \frac{1}{2}$ and $1+\frac{1}{3(1-\mu)}>1$, it induces
\[
2\left(  \beta  - \left(\frac{1}{2} + \frac{1}{3(1-\mu)} \right) \right) Q_1=0, \quad \Rightarrow \quad Q_1 =0.
\]
Iteratively, with the choice of $\beta = \frac{1}{3(1-\mu)}+ \frac{1}{2j_0}$,
\[
\beta- \frac{1}{2j} - \frac{1}{3(1-\mu)} = \frac{1}{2j_0} - \frac{1}{2j} \not =0, \quad \Rightarrow \quad Q_j=0, \quad \forall \; 2 \le j < j_0.
\]
Similarly, with $\partial_r^{2j_0} Q(0) =Q_{j_0}<0$ determined in \eqref{ODE: initial assumption},
\[
\beta- \frac{1}{2j} - \frac{1}{3(1-\mu)} \not =0, \; \Rightarrow \; Q_j =\frac{ \sum_{i=1}^{j-1} \left( \frac{2i}{2(j-i)+3} + (1-\mu)  \right)Q_i Q_{j-i} }{2j \left( \beta -\frac{1}{2j} - \frac{1}{3(1-\mu)} \right)}, \quad \forall \; j > j_0.
\]
Hence, we obtain the recurrence relation \eqref{recurrence Qj}.

We remark here that under the previous argument, if $\beta - \frac{1}{2j} - \frac{1}{3(1-\mu)} \not= 0$ for all $j \ge 1$, then $Q_j =0$ for all $j \ge 1$. In this case, $Q(r) \equiv \frac{1}{1-\mu}$, which is a constant solution and not our focus. Hence, to obtain a non-constant solution, we must give the vanishing condition that $\beta - \frac{1}{2j_0} - \frac{1}{3(1-\mu)} = 0$ for some $j_0 \ge 2$, which exactly matches the choice of $\beta$.

\mbox{}

\noindent \textit{Step 2. Analyticity of $(Q,f)$ near the origin.}
First of all,  recalling \cite[Lemma B.1]{guoMahirJangpolytropicgap}, there exists a universal constant $a>0$ such that for all $L\in\NN$,
\be
    \sum_{\substack{i+j=L \\ i,j>0}} \frac{1}{i^2 j^2} \le \frac{a}{L^2},    
    \label{combinatorics ineq}
\ee
which, by using the fact that $Lj \ge i$ for any $1 \le  i,j \le L-1$, yields that
\be\label{combinatorics ineq 1}
    \sum_{\substack{i+j=L \\ i,j>0}} \frac{1}{i j^3} \leq \sum_{\substack{i+j=L \\ i,j>0}} \frac{L}{i^2 j^2} \leq \frac{a}{L}.
\ee
Next, we are devoted to the estimates of coefficients of the Taylor expansion determined in \eqref{recurrence Qj}. Precisely, we claim that there exists $K, \alpha >0$, such that 
\be
|Q_j| \le \frac{K^{j-\alpha}}{j^2}, \quad \forall \; j \ge j_0,
\label{bound for Q_j}
\ee
where $K, \alpha>0$ satisfy
\be
|Q_{j_0}| \le \frac{K^{j_0 -\alpha}}{j_0^2} \quad \text{ and } \quad \frac{a K^{-\alpha}}{\frac{1}{2j_0} - \frac{1}{2(j_0+1)}} < \frac{1}{2}.
\label{initial Q_j}
\ee

In fact, we proceed by induction. For $1 \le j \le j_0$, it automatically holds with \eqref{recurrence Qj} and \eqref{initial Q_j}. We now assume that \eqref{bound for Q_j} holds for any $j_0 \le j \le L$, then by using \eqref{recurrence Qj}, \eqref{combinatorics ineq}, \eqref{combinatorics ineq 1}, \eqref{bound for Q_j} and \eqref{initial Q_j}, we obtain
\begin{align*}
     Q_{L+1} &= \frac{1}{2(L+1) \left( \beta - \frac{1}{2(L+1)} - \frac{1}{3(1-\mu)} \right)} \sum_{\substack{i+j=L+1 \\ i,j>0}} \left( \frac{2i}{2j+3} + (1-\mu) \right) Q_i Q_j
     \\
    &\le \frac{1}{2(L+1)\left( \frac{1}{2j_0} - \frac{1}{2(j_0+1)} \right)}\sum_{\substack{i+j=L+1 \\ i,j>0}}  \left( \frac{2i}{2j+3} + (1-\mu) \right) \frac{K^{L+1 -2 \alpha}}{i^2 j^2} \\
    & \le \frac{K^{L+1-2 \alpha} }{2(L+1)\left( \frac{1}{2j_0} - \frac{1}{2(j_0+1)} \right)}\sum_{\substack{i+j=L+1 \\ i,j>0}} \left( \frac{1}{ij^3} + \frac{1}{i^2j^2} \right) \le \frac{K^{L+1-\alpha}}{(L+1)^2},
\end{align*}
which yields the validity of \eqref{bound for Q_j} with $j=L+1$, hence we have closed the induction and verified the claim.

The analyticity of $Q$ at the origin then follows from Cauchy's convergence criterion for power series. Precisely, there exists $\epsilon = \epsilon(K) \ll 1$ such that the series
\[
Q(r) =  \sum_{j=0}^\infty Q_{j} r^{2j}, \quad \text{ with } Q_j \text{ determined in \eqref{recurrence Qj},} 
\]
is absolutely convergent on $[0,\epsilon]$. The analyticity of $f(r) =\sum_{j=0}^\infty f_j r^{2j}$ on $r \in [0,\epsilon]$ simply follows \eqref{recurrence relation 1} and the analyticity on $[0,\epsilon]$ of $Q$.

\mbox{}

\noindent \textit{Step 3. Extend $(Q,f)$ to $[0,+\infty)$.} 
At this stage, our objective is to extend the previously constructed solution to $\overline{\RR^+}$, and verify that both $f(r)$ and $Q(r)$ will asymptotically decay to $0$ and the solution curve (see the red curve in Figure \ref{fig3})) remains within $\calM$ with
\[ \mathcal M :=\left\{(Q,f) \Big| \;  0< Q<Q_0, \; \frac{Q}{3}<f<\beta\right\}.\]
Here, recalling \eqref{ODE system} or Figure \ref{fig3}, we observe that both $(Q(r)$ and $f(r))$ are strictly decreasing once the solution curve stays within $\calM$, i.e.
\be 
Q'(r) < 0, f'(r)<0 \quad \text{ if} \quad (Q(r),f(r)) \in \text{ $\mathcal M$}.
\label{M: decreasing}
\ee
We first claim that there exists $\nu \in (0,\epsilon)$ such that $(Q(r),f(r)) \in \calM, \; \forall \; r \in (0,\nu]$. To verify this, since
\[
\frac{d^j}{dr^j} Q(0) = 0, \quad \forall \; 1 \le j < 2j_0  \qquad \text{ and } \qquad
\frac{d^{2 j_0}}{dr^{2j_0}} Q(0) = (2j_0)! Q_{j_0} <0,
\]
it follows that both $Q(r)$ and $f(r)$ are strictly decreasing on $[0,\nu]$ for some $0<\nu \le  \epsilon \ll 1$. Consequently, the portion of the solution curve for $r \in [0,\nu]$, which begins at $P_0$ \eqref{P_0: def}, remains within the lower-left area of $P_0$. Furthermore, it continues to be representable as a graph, meaning that there is a smooth function $\calF$ with 
\[
f = \calF(Q)  \quad \text{ on } \quad  Q \in \left[Q_0 - \nu_Q, Q_0 \right],
\]
for some $0 <\nu_Q \ll 1$. By chain rule, \eqref{ODE system}, \eqref{(Q,f), expansion}, \eqref{recurrence relation 1}  \eqref{recurrence Qj} and $j_0 \ge J >1$,
\begin{align*}
    \mathcal F'(Q_0)
    &=\frac{df}{dQ}\Big|_{(Q,f)=P_0}
    = \lim_{r \to 0^+}\frac{(Q(r)-3f(r))(\beta-f(r))}{(1-\mu)Q^2(r)-Q(r)} 
    \\
    &
    = \frac{\beta - f_0}{(1-\mu) Q_0} \frac{Q_{j_0} -3f_{j_0}}{Q_{j_0}}
    = \frac{1}{2j_0} \left( 1- \frac{3}{2j_0 +3} \right) < \frac{1}{2j_0} \le \frac{1}{4} < \frac{1}{3}.
\end{align*}
Thus, the slope of the solution curve at $P_0$ is strictly less than that of the line $f= \frac{1}{3} Q$, which contains the lower boundary of $\calM$. So by adjusting $0< \nu \ll 1$ if necessary and using the smoothness of $\calF$, we conclude that the curve $\{(Q(r), f(r))\}_{r \in(0,\nu]}$ lies above $f = \frac{1}{3} Q$. This has verified the claim.
\vspace{2mm}

Next, we prove that the solution curve remains in $\mathcal M$ for all $r\in(0,\infty)$. In fact, ensured by the Cauchy-Lipschitz theory, with the strict decay property of the solution within $\calM$ (see \eqref{M: decreasing}), there are only three possible scenarios:
\begin{enumerate}
    \item there exists $r_{E1} \in (0,\infty]$, such that the solution exits the region $\mathcal M$ at a point $\left( Q_{E1},\frac{Q_{E1}}{3} \right)$ through the line $f=\frac{Q}{3}$, where $Q_{E1} = Q(r_{E1})\in \left(0, Q_0 \right)$;
    \item there exists $r_{E2} \in (0,\infty]$, such that the solution exits the region $\mathcal M$ at a point $(0,f_{E2})$ through the $f$-axis, where $f_{E2} = f(r_{E2})\in(0,f_0)$;
    \item the solution escapes $\mathcal M$ at the origin $O$.
\end{enumerate}

We now eliminate scenarios $(1)$ and $(2)$ by contradiction arguments. 

For the scenario $(1)$, we assume that it occurs. Since $(Q(r),f(r))$ is strictly decreasing before reaching the line $f=\frac{Q}{3}$, together with Cauchy-Lipschitz theory, the solution curve can be further extended on $Q\in[Q_{E1},Q_0]$ and represented as $Q \mapsto \calF(Q)$ with $\calF$ a smooth function.

On the one hand, since $\{(Q(r),f(r))\}_{r \in (0,r_{E_1})} \subset \calM $, it follows that $f(r) > \frac{1}{3} Q(r)$ for any $r \in (0,r_{E_1})$ and $f(r), Q(r)$ is strictly decreasing on $r \in (0,r_{E_1})$, thus together with $f(r_{E_1}) = \frac{1}{3}Q(r_{E_1})$,
\[
\calF'(Q_{E1})
= \lim_{r \to r_{E_1}^-} \frac{f(r) - f(r_{E_1})}{Q(r)-Q(r_{E_1})}
\ge \lim_{r \to r_{E_1}^-} \frac{\frac{1}{3}Q(r) - \frac{1}{3} Q(r_{E_1})}{Q(r) - Q(r_{E_1})} = \frac{1}{3}.
\]

On the other hand, recalling the ODE system \eqref{ODE system} together with $f(r_{E_1}) = \frac{1}{3}Q(r_{E_1})$ again, we compute
\begin{align*}
    \mathcal F'(Q_{E1}) = 
    \lim_{r \to r_{E1}^-}\frac{(Q(r)-3f(r))(\beta-f(r))}{(1-\mu)Q^2(r)-Q(r)} =0< \frac{1}{3},
\end{align*}
which contradicts our assumption and hence we have ruled out the scenarios $(1)$.

Regarding the second scenario $(2)$, by an argument analogous to that of the first case, the solution curve can be extended on $f \in [f_{E2}, f_0]$ and represented as $f \mapsto \calQ(f)$ with $\calQ$ a smooth function, and thus
\be
\{ \left(Q(r), f(r) \right) \}_{r \in \big(0,r_{E_2} \big)} 
=\{ \left( \calQ(f), f \right) \}_{f \in \big(f_{E_2}, f_0 \big)}
\subset \calM.
\label{case 2: one hand}
\ee
Recall \eqref{ODE system} again, $\calQ(f)$ is governed by the equation
\be
\calQ'= \frac{(1-\mu) \calQ^2 - \calQ}{(\calQ -3f) (\beta -f)}, \quad \forall\; f  \in  [f_{E2}, f_{0}],
\label{case 2: ODE}
\ee
whose local wellposedness near $f=f_{E_2}$ is well established by the Cauchy-Lipschitz theory. Nevertheless, besides the nontrivial solution curve given in \eqref{case 2: one hand}, we observe that $\{ (0, f)\}_{f\in [f_{E_2}, f_0]}$ is also a satisfied solution curve, which leads to a contradiction and hence we have ruled out the scenario $(2)$.

Consequently, we have shown that the solution curve could only exist $\mathcal M$ at the origin $O$, i.e. in scenario $(3)$. To conclude this step, it suffices to verify that 
\be
\lim_{r\to \infty}(Q(r),f(r))=(0,0).
\label{Step 3: decay of Q,f}
\ee
In fact, similar to the previous argument, the solution curve can be further smoothly extended and represented as $Q \mapsto \calF(Q)$ on $[0,Q_0]$. Additionally, using \eqref{M: decreasing} again, the function $r \mapsto Q(r)$ is one-to-one and thus $r$ can be expressed by $r= \calR(Q)$ on $Q \in [0,Q_0]$, with $\calR$ solving
\[
    \mathcal R'(Q) = \mathcal R(Q) \frac{\beta-\mathcal F(Q)}{(1-\mu) Q^2- Q},
\]
which follows from the inverse function theorem and \eqref{ODE system}. Then, a direct computation together with \eqref{ODE: initial assumption}, \eqref{recurrence relation 1}, \eqref{M: decreasing} and the fact that $\mathcal R(\frac{Q_0}{2})\in(0,\infty)$ yields that 
\begin{align*}
    \lim_{Q \to 0^+}\ln \mathcal R(Q) &=\ln \mathcal R\left(\frac{Q_0}{2} \right) +\lim_{Q \to 0^+}\int_{\frac{Q_0}{2}}^{Q} \frac{\beta-\mathcal F(\tilde Q)}{(1-\mu)\tilde Q^2-\tilde Q} \ d\tilde Q \\
    &> \ln \mathcal R\left(\frac{Q_0}{2}\right) +(\beta - f_0)\lim_{Q \to 0+}\int_{\frac{Q_0}{2}}^{Q} \left( \frac{1}{\tilde Q-\frac{1}{1-\mu}}-\frac{1}{\tilde Q} \right) \ d\tilde Q 
    = + \infty,
\end{align*}
and thus we have verified \eqref{Step 3: decay of Q,f} and finished this step.

\mbox{}

\noindent \textit{Step 4. Uniqueness.} Assume that there exist two smooth solutions $(Q^{(1)},f^{(1)})$ and $(Q^{(2)},f^{(2)})$ solving the system \eqref{ODE system} with the same initial conditions \eqref{ODE: initial assumption}, then by the $C^\infty$ continuity of $Q$ at the origin and a similar argument as in \textit{Step 1} and \textit{Step 3}, it can be verified that $(Q^{(i)},f^{(i)})$ satisfies
\begin{align}\label{condition to unique}
    Q_0^{(1)} = Q_0^{(2)} = \frac{1}{1-\mu},
\quad 
Q_j^{(1)} = Q_j^{(2)} = 0 \quad \left( \text{with } 1 \le j < j_0-1 \right), \quad 
Q_{j_0}^{(1)} = Q_{j_0}^{(2)}  < 0,
\end{align}
and both $Q^{(i)}(r)$ and $f^{(i)}(r)$ are strictly decreasing on $\RR^+$ for both $i=1,2$. Next, we study the difference between these two solutions. Denote
\[
(\triangle Q,\triangle f) := (Q^{(1)}-Q^{(2)}, f^{(1)} -f^{(2)}),
\]
which solves the following ODE system:
\[
\begin{cases}
    (\triangle Q)' = \frac{1}{\beta -f^{(1)}} \frac{(1-\mu)(Q^{(1)}+Q^{(2)}) \triangle Q - \triangle Q}{r} - \frac{(1-\mu) (Q^{(2)} )^2 -Q^{(2)}}{r} \frac{\triangle f}{(\beta-f^{(1)})(\beta-f^{(2)})}, \vspace{0.1cm}\\
    (\triangle f)' = \frac{\triangle Q - 3\triangle f}{r}.
\end{cases}
\]
We define
\[
\calE(r)  : = |\triangle Q(r)|^2 + |\triangle f(r)|^2,
\]
and compute
\begin{align*}
         \;\; \calE' &= I+II, \\
      \text{ with } \ \ I &= 2 \frac{\triangle f \triangle Q -3 (\triangle f)^2}{r} 
    + 2\frac{(1-\mu) (Q^{(1)}+Q^{(2)}) -1}{(\beta -f^{(1)})r} (\triangle Q)^2, \\
        \; II&=  -2 \frac{(1-\mu) (Q^{(2)})^2 -Q^{(2)}}{r} \frac{1}{(\beta -f^{(2)})(\beta-f^{(1)})} \triangle f \triangle Q,
     \end{align*}
By using Cauchy inequality, strictly decreasing fact of $Q^{(i)}(r)$ and $f^{(i)}(r)$, \eqref{ODE: initial assumption} and \eqref{recurrence relation 1}, we bound $I$ by
\begin{align*}
    I \leq  \Big(1+ 2\frac{(1-\mu) (Q^{(1)}(0)+Q^{(2)}(0)) -1}{(\beta -f^{(1)}(0))}\Big) \frac{\mathcal E(r)}{r} = (4j_0+1)\frac{\mathcal E(r)}{r}.
\end{align*}
Applying condition \eqref{condition to unique} along with the Cauchy's inequality, there exists constant $M= M(j_0,\mu, Q^{(1)}, Q^{(2)})>0$ such that
\begin{align*}
    II \le  M |\triangle f \triangle Q| \leq M \calE(r),
\end{align*}
which yields that
\[
    \calE'(r) = I+II \leq (4j_0 + 1) \frac{\calE(r)}{r} + M \calE(r).
\]
Applying Gronwall's inequality,
\begin{align*}
    0 \le \calE(r) \le  \calE(\delta)\left( \frac{r}{\delta} \right)^{4 j_0 +1}  e^{M(r-\delta)}, \quad \text{ uniformly for } r,\delta>0.
\end{align*}
Since $(Q^{(i)},f^{(i)})$ (with $i=1,2$) satisfy the same initial conditions \eqref{condition to unique}, it follows that
\[
\partial_r^{j} \calE(r) \Big|_{r=0} =0, \quad \forall \; 0 \le j \le 4j_0+2,
\]
which implies that
\[
\lim_{\delta \to 0+} \frac{\calE(\delta)}{\delta^{4j_0+1}} e^{-M \delta} = 0.
\]
By the squeezing theorem, it follows that $\calE(r) \equiv 0$ for any $r \ge 0$. Consequently, we conclude $Q^{(1)} \equiv Q^{(2)}$ and this completes the uniqueness.

\mbox{}

\noindent \textit{Step 5. Decay estimate of $Q$ and its derivative.} 
Based on the discussion in \textit{Step 3}, both $Q(r)\searrow 0$ and $f(r)\searrow 0$ as $r\to \infty$. Next for any fixed $\nu_0>0$, recalling the choice of $\beta= f_0+ \frac{1}{2j_0}$, there exists $M_0 = M_0(\nu_0)\gg 1$ such that
\begin{align*}
    & \qquad \quad 0<(1-\nu_0)Q(r)\leq Q(r)-(1-\mu)Q^2(r)\leq Q(r), \quad \forall \; r \ge M_0, \\
    & \text{ and } \; \; -2j_0\leq -\frac{1}{\beta-f}\leq -\frac1\beta <0.
\end{align*}
From the ODE \eqref{ODE system} satisfied by $Q(r)$, these inequalities imply
\[
- \frac{2j_0Q}{r} \le Q'(r) \le -\frac{(1-\nu_0) Q}{\beta r}, \quad \quad \forall \; r \ge M_0,
\]
hence
\[
Q(r) \le  Q(M_0) \left( \frac{r}{M_0} \right)^{-\frac{1}{\beta}(1-\nu_0)}, \quad \forall \; r \ge M_0.
\]
In particular, since $\beta < \frac{1}{2}$, we can choose $\nu_0 >0$ sufficiently small such that $\frac{1-\nu_0}{\beta}< 2$. This allows us to find a constant $C=C(\beta,M_0) > 0$ such that
\be
Q(r) \le C \la r \ra^{-2}, \quad \forall \; r \ge 0.
\label{Lemma 2.2 decay 1}
\ee
Moreover, using the second equation in \eqref{ODE system}, there exists a constant $C= C(\beta,M_0)>0$ such that
\be
0< f(r) = \frac{1}{r^3} \int_0^r Q(s) s^2 ds \le \frac{C}{r^3} \int_0^r \frac{s^2}{\la s \ra^2} ds \le C \la r \ra^{-2}, \quad \forall \; r \ge 0.
\label{lemma 2.2 decay 2}
\ee
For the derivatives of $Q(r)$ and $f(r)$, applying $\partial_r^{j-1}$ (with $j \ge 1$) to \eqref{ODE system} and using the base decay estimates \eqref{Lemma 2.2 decay 1} and \eqref{lemma 2.2 decay 2}, a straightforward induction then shows that
\begin{align}
|\partial_r^{(j)}Q(r)| \le C \la r \ra^{-2-j},
\label{asym Q higer order}
\end{align}
for some constant $C=C(\beta,M_0,j)>0$.

Finally, combining the $C^\infty$ smoothness of $Q$, the above decay bound on $\partial_r^{(j)}Q$ (cf. \eqref{Lemma 2.2 decay 1} and \eqref{asym Q higer order}), and the standard identity $D g= \frac{x}{|x|} \cdot \na g$ for radial function $g$, we can verify that $Q \in H^\infty(\RR^3)$. This completes the proof of the lemma.
\end{proof}

Based on Lemma \ref{lem: solve ODE}, one naturally obtains the existence of a smooth profile satisfying \eqref{PDE Q}. More precisely, we have the following result.

\begin{corollary}[Existence of smooth self similar profile]
\label{thm: existence of the profile}
Under the assumptions of Lemma \ref{lem: solve ODE}, for any fixed $[Q]_{j_0} <0$, there exists a unique smooth radial symmetric solution $Q\in H^\infty(\RR^3)$ to \eqref{PDE Q} satisfying
    \be
    Q(0) = \frac{1}{1-\mu}, \quad \text{ and } \partial_r^{(2j_0)} Q(0) = (2j_0)! [Q]_{j_0} <0.
    \label{smooth profile: initial data}
    \ee
    In particular, $Q(r)$ is strictly radially decreasing with respect to $r \ge 0$ and satisfies the decay estimates \eqref{decay estimates of Q}.
\end{corollary}

\section{Linear theory}
\label{sec: linear theory}
Starting from this section, we fix $\mu\in[0,\frac{1}{3})$ and a positive integer $j_0\geq \frac{3(1-\mu)}{1-3\mu} +1>1$, then we have
\be
\beta =  \frac{1}{3(1-\mu)} + \frac{1}{2j_0}\in \left(0,\frac{1}{2} \right).
\label{beta: def}
\ee
We now introduce the self-similar coordinate:
\be
 y =\frac{x}{\lambda^{2\beta}}, \quad \frac{d\tau}{dt} = \frac{1}{\lambda^2}, \quad 
 \tau\Big|_{t=0} = 0, \quad \frac{\lambda_\tau}{\lambda} = -\frac{1}{2},
 \label{ss coordinte}
\ee
and the corresponding renormalization
\be
\rho(t,x) = \frac{1}{\lambda^2} \Psi \left( \tau, y \right).
\label{ss renormalization}
\ee
The equation \eqref{equation: KS with damping} can be mapped into the renormalization system of $\Psi$:
\be
\partial_\tau \Psi = \lambda^{2-4\beta} \Delta \Psi - \Psi - \beta y \cdot \na \Psi + \na \Psi \cdot \na \Dein \Psi + (1-\mu) \Psi^2,
\label{renormalization KS}
\ee
where, from \eqref{ss coordinte}, the function $\lambda$ is
\be
\lambda(\tau) = \lambda_0 e^{-\frac{1}{2} \tau}, \quad \forall \; \tau \ge 0,
\label{lambda: form}
\ee
with initial data given by $\lambda \Big|_{\tau=0}=\lambda_0>0$. 

Next, we consider a solution to the above equation in perturbative form
\[
    \Psi = Q+ \ep,
\]
where $Q$ is a solution to the PDE \eqref{PDE Q}, corresponding to the fixed parameters $\mu$ and $j_0$, whose existence is established in Lemma \ref{lem: solve ODE}. Direct computation shows 
the residue $\ep$ solves
\be
\partial_\tau \ep  =  \calL \ep + \lambda^{2-4\beta} \Delta \Psi + N(\ep),
\label{linearized eq}
\ee
where $\calL$ is the linear operator
\be
\calL \ep = -(\ep + \beta y \cdot \na  \ep) + \na \ep \cdot \na \Dein Q + \na Q \cdot \na \Dein \ep + 2(1-\mu) Q \ep,
\label{linearized op}
\ee
and $N(\ep)$ is the nonlinear term
\be
N(\ep) = \na \ep \cdot \na \Dein \ep  + (1-\mu) \ep^2.
\label{nonlinear term}
\ee
For the linear operator $\calL$ given in \eqref{linearized op}, we will prove its coercivity in the sense of $L_w^2$ for some weight $w$ sufficiently singular at the origin. Precisely, we have the following result:
\begin{proposition}[Coercivity of $\calL$ in $L_w^2$]
\label{prop: coer}
    For $\calL$ given in \eqref{linearized op}, there exists a weight defined by
    \be
    w(y) :=  |y|^{-A} + B,
    \label{singular weight}
    \ee
    where $(A,B) \in \ZZ_+ \times \RR_+$ satisfy $A \gg 4 j_0 +3$ and $\frac{A}{4} \in  \ZZ_{>0}$, such that $\calL$ satisfies the coercivity in the following sense:
    \be
    \left( \calL g,g \right)_{L_w^2} \le - \frac{1}{8} \|g\|_{L_w^2}^2, \quad \forall \; g \in L_w^2(\RR^3) \quad \text{radially symmetric}.
    \label{L: coer}
    \ee
\end{proposition}

\begin{remark}
    Due to the limited quantitative properties of $Q$ available from Lemma \ref{lem: solve ODE} and the nonlocal structure of the operator $\mathcal{L}$, it is difficult to identify the optimal value of $A$ explicitly, in contrast with the situation for the heat equation discussed in Section \ref{sec:slh}. Nevertheless, the existence of some $A$ satisfying \eqref{L: coer} is sufficient for our subsequent analysis.
\end{remark}

\begin{proof}[Proof of Proposition \ref{prop: coer}]
    First of all, for any $g \in L_w^2$, by integration by parts, 
    \eqref{nonlocal term: formula} and the fact that $y \cdot \na w = -A |y|^{-A}$,
    \begin{align*}
      -\beta \int (y \cdot \na g) g w dy
        = \frac{\beta}{2}\left( -A + 3 \right)\int g^2 |y|^{-A} dy 
        +\frac{3B \beta}{2} \int g^2 dy,
    \end{align*}
    and
    \begin{align}
    \int \left( \na g \cdot \na \Dein Q \right) g w dy 
     = -\frac{1}{2} \int Q g^2 w dy + \frac{A}{2} \int g^2 f_Q |y|^{-A}dy,
     \label{coercivity: IBP method}
    \end{align}
    with which, recalling the definition of $\calL$ ( see \eqref{linearized op}),  we split $(\calL g,g)_{L_w^2}$ into the following three parts:
    \be
    (\calL g,g)_{L_w^2}
    = I_{SI} + I_{LO} + I_{NLO},
    \label{(Lg,g) L2 form}
    \ee
    where
    \begin{align*}
        & I_{SI} : =\left( -1 + \frac{\beta}{2} (-A+3) \right) \int g^2 |y|^{-A} dy + \frac{A}{2} \int g^2 f_Q |y|^{-A}dy + \left( \frac{3}{2} -2 \mu \right) \int g^2 Q |y|^{-A} dy , \\
        & I_{LO} :=\left( -B + \frac{3B \beta}{2} \right) \int g^2 dy + \left(\frac{3}{2}-2\mu \right) B \int g^2 Q dy, \quad I_{NLO}:=\int \left( \na Q \cdot \na \Dein g \right) g w dy.
    \end{align*}
     Here $I_{SI}$ denotes the sum of local terms involving the singular weight $|y|^{-A}$, $I_{LO}$ denotes the sum of local terms without singular weight, and $I_{NLO}$ denotes the sum of nonlocal terms.

   \mbox{}
   
    \noindent \underline{\textit{Estimate of $I_{SI}$}.} 
    By Lemma \ref{lem: solve ODE},  $Q$ and $f_Q$ are both radially decreasing. Moreover, recalling the initial conditions of $(Q(0),f_Q(0))$ given in \eqref{ODE: initial assumption} and the definition of $\beta$ in \eqref{beta: def}, $I_{SI}$ can be estimated by
\begin{align}
    I_{SI} 
     & \le \left( -1 + \frac{\beta}{2} (-A+3) \right) \int g^2 |y|^{-A} dy 
     + \left( \frac{A}{2} f_Q(0) + \left(\frac{3}{2}-2\mu \right) Q(0) \right) \int g^2 |y|^{-A} dy \notag\\
     & = \left( -1 + \frac{\frac{1}{3(1-\mu)}+ \frac{1}{2j_0}}{2} (-A+3) + \frac{A}{6(1-\mu)} + \left(\frac{3}{2}-2\mu \right) \frac{1}{1-\mu} \right) \int g^2 |y|^{-A} dy \notag\\
     & = \left( 1 + \frac{-A+3}{4j_0} \right) \int g^2 |y|^{-A} dy,
     \label{coer: singularity weight estimate 1}
\end{align}
where the coercivity becomes linearly stronger as $A \gg 1$ grows.

We remark here that the derivation of \eqref{coer: singularity weight estimate 1} is rather delicate. Precisely, one might naturally expect that the term $g + \beta y \cdot \na g$ alone provides enhanced coercivity with increasing $A$. However, this gain is offset by the contribution $\frac{A}{2} \int g^2 f_Q |y|^{-A} dy$, which greatly diminishes the coercivity originating from the scaling term. Nevertheless, with the estimate given in \eqref{coer: singularity weight estimate 1}, we find that this essential obstacle can be almost exactly controlled via the fact that $\beta -f_Q(r) \ge \beta -f_Q(0) = \frac{1}{2j_0}>0$.

\mbox{}

\noindent \underline{\textit{Estimate of $I_{LO}$}.} Again, by Lemma \ref{lem: solve ODE}, $Q$ is radially decreasing and $\Big\| |y|^{-\frac{1}{2}} \partial_r Q  \Big\|_{L^2} <\infty$, so there exists $R_1 = R_1(Q) \gg 1$ such that
\be
\frac{3}{2} Q(R_1) \le \frac{1}{1000} \quad \text{ and } \quad 
\Big\| |y|^{-\frac{1}{2}} \partial_r Q  \Big\|_{L^2(|x| \ge R_1)} \le \frac{1}{5000}.
\label{R1: choice linear sta}
\ee
Noting that $\beta< \frac{1}{2}$, we may now estimate the lower-order term $I_{LO}$ as follows:
    \begin{align}
        I_{LO} 
        & \le  \left( -B + \frac{3}{4} B \right) \int g^2 dy + \left(\frac{3}{2}-2\mu \right) B  \left( \int_{|y| \le R_1} g^2 Q dy   + \int_{|y| \ge R_1} g^2 Q dy \right) \notag \\
        & \le -\frac{B}{4} \int g^2 dy + \left( \frac{3}{2} -2\mu \right)B Q(0) R_1^A \int_{|y| \le R_1} g^2 |y|^{-A} dy
        + \left( \frac{3}{2} -2 \mu \right) B  Q(R_1) \|g \|_{L^2}^2 
        \notag \\
        & = -\frac{3B}{16}  \| g \|_{L^2}^2 +  \left( \frac{3}{2} -2 \mu \right)  B Q(0) R_1^A \| g\|_{L_{|y|^{-A}}^2}^2.
        \label{I LO}
    \end{align}
   Given the limited qualitative information about $Q$, the term involving $Q$ cannot be directly controlled, especially the integral of $\int_{|y| \le R_1}\int g^2 Qdy$, where $Q(y) = O(1)$ on $B(0,R_1)$. To address this, we will absorb this term using the previous coercivity \eqref{coer: singularity weight estimate 1} by choosing $0< B\ll 1$ sufficiently small.

\mbox{}

\noindent \underline{\textit{Estimate of $I_{NLO}$}.} By Cauchy's inequality, the nonlocal term involving in $I_{NLO}$ can be  controlled pointwise by
\begin{align}
    & \quad \frac{1}{r^2} \int_0^r |g(s)| s^2 ds
     = \frac{1}{ 4\pi r^2} \int_{B(0,r)} |g(y)| dy = \frac{1}{4 \pi r^2}\int_{B(0,r)} |g(y)| |y|^{-\frac{A}{2}} |y|^{\frac{A}{2}} dy \notag\\
    & \le \frac{1}{4 \pi r^2} \left( \int_{B(0,r)} g^2(y) |y|^{-A} dy \right)^\frac{1}{2} \left( \int_{B(0,r)}  |y|^{A} dy \right)^\frac{1}{2}  = \sqrt{\frac{1}{4 \pi(A+3)}} r^{\frac{A}{2} - \frac{1}{2}} \| g \|_{L_{|y|^{-A}}^2},
    \label{linear estimate: nonlocal 1}
\end{align}
and this induces that
\begin{align}
    & \quad \Bigg| \int \left( \na Q \cdot \na \Dein g \right) g |y|^{-A} dy \Bigg|
     = \int |\partial_r Q| \left(  \frac{1}{|y|^2}  \int_0^{|y|} |g(s)| s^2 ds \right) |g||y|^{-A} dy \notag\\
    & \le \sqrt{\frac{1}{4 \pi(A+3)}} \| g \|_{L_{|y|^{-A}}^2}  \int |\partial_rQ|  |g| |y|^{-\frac{A}{2} - \frac{1}{2}} dy 
     \le  \sqrt{\frac{1}{4 \pi(A+3)}} \Big\| |y|^{-\frac{1}{2}} \partial_r Q \Big\|_{L^2} \| g \|_{L_{|y|^{-A}}^2}^2,
     \label{I NLO 1}
    \end{align}
which, by choosing $A \gg 1$ sufficiently large, can be absorbed in \eqref{coer: singularity weight estimate 1}. Similarly, repeating the previous argument, together with the choice of $R_1 \gg 1$ (see \eqref{R1: choice linear sta}), the $L^2$ component of $I_{NLO}$ can be controlled by:
    \begin{align}
         & \quad B\int \left( \na Q \cdot \na \Dein g \right) g dy  
        \le B \| g \|_{L^2} \left(  \int_{|y| \le R_1} |\partial_r Q| |y|^{-\frac{1}{2}} |g| dy + \int_{|y| \ge R_1} |\partial_r Q| |y|^{-\frac{1}{2}} |g| dy  \right) \notag \\
        & \le \frac{B}{100} \| g \|_{L^2}^2 + 50 B R_1^A \big\| \partial_r Q |y|^{-\frac{1}{2}} \big\|_{L^2}^2 \| g \|_{L_{|y|^{-A}}^2}^2 + 50 B \big\| \partial_r Q |y|^{-\frac{1}{2}} \big\|_{L^2(|y| \ge R_1)}^2 \| g \|_{L^2}^2 \notag \\
        & \le \frac{B}{50} \| g \|_{L^2}^2 
        + 50 B R_1^A \big\| \partial_r Q |y|^{-\frac{1}{2}} \big\|_{L^2}^2 \| g \|_{L_{|y|^{-A}}^2}^2.
        \label{I NLO 2}
    \end{align}
    The idea of handling this term is similar to $I_{LO}$, and we can use the smallness of $0< B \ll 1$ to absorb $B \int_{|y| \le R_1} \left( \na Q \cdot \na \Dein g \right) g dy$ by the coercivity \eqref{coer: singularity weight estimate 1}.
    Hence combining \eqref{I NLO 1} and \eqref{I NLO 2} yields
    \begin{align}
        I_{NLO}
        & \le  \left( \sqrt{\frac{1}{4 \pi(A+3)}} \big\| |y|^{-\frac{1}{2}} \partial_r Q \big\|_{L^2}  + 50 B R_1^A \big\| |y|^{-\frac{1}{2}} \partial_r Q  \big\|_{L^2}^2  \right)\| g \|_{L_{|y|^{-A}}^2}^2 
        + \frac{B}{50} \| g \|_{L^2}^2.
        \label{I NLO}
    \end{align}

    \mbox{}

    \noindent \underline{\textit{Conclusion of coercivity}.} Combining the estimates \eqref{coer: singularity weight estimate 1}, \eqref{I LO} and \eqref{I NLO}, we obtain that
    \begin{align*}
        \left( \calL g, g \right)_{L_w^2}
        & \le 
        -\frac{1}{8} B \| g \|_{L^2}^2
        +  \left( 1 + \frac{-A+3}{4j_0}  + \sqrt{\frac{1}{4 \pi(A+3)}} \Big\| |y|^{-\frac{1}{2}} \partial_r Q \Big\|_{L^2} \right) \int g^2 |y|^{-A} dy \\
        & \quad + \left( \left( \frac{3}{2} -2 \mu \right) B R_1^A Q(0)  + 50 B R_1^A \big\| |y|^{-\frac{1}{2}} \partial_r Q  \big\|_{L^2}^2  \right)\| g \|_{L_{|y|^{-A}}^2}^2,
    \end{align*}
    with $R_1 = R_1(Q) >0$ chosen as in \eqref{R1: choice linear sta}. Next, we choose 
    $A \gg 4j_0 +4$ with $\frac{A}{4} \in \ZZ_{>0}$ such that
    \begin{align}
        1 + \frac{-A+3}{4j_0} \le -1, \quad \text{ and } \quad 
        \sqrt{\frac{1}{4 \pi(A+3)}} \Big\| |y|^{-\frac{1}{2}} \partial_r Q \Big\|_{L^2}
        \le  \frac{1}{100},
        \label{A: choice range}
    \end{align}
    then we choose $0< B=B(A,R_1,\mu) \ll 1$ sufficiently small such that
    \[
    \left( \frac{3}{2} -2 \mu \right) B R_1^A Q(0) +  50 B R_1^A \big\| |y|^{-\frac{1}{2}} \partial_r Q  \big\|_{L^2}^2  \le \frac{1}{100},
    \]
   so that we obtain that
    \begin{align*}
         \left( \calL g, g \right)_{L_w^2}
         \le -\frac{1}{8} \| g \|_{L_{|y|^{-A}}^2}^2
         - \frac{1}{8} B \| g\|_{L^2}^2 = -\frac{1}{8} \| g \|_{L_w^2}^2,
    \end{align*}
    which completes the proof of \eqref{L: coer}.
\end{proof}

In addition, to study the nonlinear stability of $Q$ in Section \ref{sec: nonlinear sta}, it is necessary to establish the coercivity properties of the linearized operator $\calL$ in a higher order Sobolev space $\dot H^{2K+4}$ with $K = \frac{A}{4}$. The precise statement is as follows:

\begin{proposition}[Coercivity of $\calL$ in $\dot H^{2K+4}$]
\label{prop: core H^K}
    Let $A >0 $ be as given in Proposition \ref{prop: coer}, and define $K := \frac{A}{4} \in \ZZ_{\ge 1}$. Then there exists constant $C=C(Q,A,B,j_0,\mu)>0$ such that
    \be
    (\calL g ,g)_{\dot H^{2K+4}} 
    \le - \frac{1}{8} \| g \|_{\dot H^{2K+4}}^2 + C \| g \|_{L^2}^2, \quad \forall \; g \in H_{rad}^{2K+4}.
    \label{coercivity in H^K}
    \ee
\end{proposition}

\begin{proof}
    By the definition of $\calL$ given in \eqref{linearized op},
    \be
    \left( \calL g,g \right)_{\dot H^{2K+4}} 
    = I + II + III + IV,
    \label{HK estimate: expression}
    \ee
    with
    \begin{align*}
        &  I: = -\left( \Delta^{K+2} \left( (g + \beta y \cdot \na  g)  \right),  \Delta^{K+2} g\right)_{L^2}, 
        \quad II:= \left( \Delta^{K+2} \left( \na g \cdot \na \Dein Q \right),  \Delta^{K+2} g\right)_{L^2}, \\
        & III:=\left( \Delta^{K+2} \left( \na Q \cdot \na \Dein g  \right),  \Delta^{K+2} g\right)_{L^2}, 
        \quad IV := \left( \Delta^{K+2} \left(  2(1-\mu) Q g \right),  \Delta^{K+2} g\right)_{L^2}.
    \end{align*}
    First, observe that by scaling invariance, the term $I$ can be computed as
    \begin{equation}\label{5.2.1}
    \begin{aligned}
    I &= 
    -\frac{1}{2}\frac{d}{d\lambda}\Bigg|_{\lambda =1} \Big\| \Delta^{K+2} \left( \lambda g(\lambda^\beta y) \right)\Big\|_{L^2}^2
    = -\left( \frac{\beta}{2} \left( 4K+5 \right)  +1 \right) \| \Delta^{K+2} g \|_{L^2}^2.
    \end{aligned}
    \end{equation}

    For the term $II$, recalling \eqref{nonlocal term: formula}, we rewrite $\na \Dein Q = rf_Q \mathbf{e}_r$ with $\mathbf{e}_r = \frac{y}{|y|}$. Then by Lemma \ref{lem: solve ODE}, \cite[Lemma A.$4$]{Javierblowup3DNS}, integration by parts, Gagliardo-Nirenberg inequality, and Young's inequality, there exists a uniform constant $C= C(K,Q)>0$ such that
    \begin{align}
        II
         \leq  &\left( \left( \na  \Delta^{K+2}g \cdot \na \Dein Q \right), \Delta^{K+2}g \right)_{L^2}
        + 2(K+2) \left( (\partial_r (rf_Q))  \Delta^{K+2}g , \Delta^{K+2}g \right)_{L^2} \notag \\
        &+C \big\| r f_Q \big\|_{W^{2K+4,\infty}} \| g \|_{H^{2K+3}} \| g \|_{\dot H^{2K+4}}   \notag \\
         \leq& 
          2(K+2) \left( (\partial_r(rf_Q) )  \Delta^{K+2}g , \Delta^{K+2}g \right)_{L^2} 
      + \frac{1}{100} \norm{g}_{\dot H^{K+2}}^2 + C \norm{g}_{L^2}^2.
      \label{Hk estimate: 1}
    \end{align}
   Since $Q(r)$ is radially decreasing by Lemma \ref{lem: solve ODE},
    \begin{align*}\label{3fQ>Q}
    f_Q(r) = \frac{1}{r^3} \int_0^r Q(s) s^2 ds \ge \frac{1}{r^3} Q(r) \int_0^r s^2 ds = \frac{1}{3} Q(r),
    \end{align*}
  which yields that
    \[
    \partial_r (r f_Q(r)) = Q(r) -2f_Q(r) \le 3f_Q(r) -2 f_Q(r) = f_Q(r), \quad \forall \; r >0.
    \]
    Hence, by the decreasing fact of $f_Q$ given by Lemma \ref{lem: solve ODE}, \eqref{Hk estimate: 1} can be improved to
    \begin{align}
        II & \le 
       \left(  2(K+2) f_Q(0) + \frac{1}{100} \right) \| g \|_{\dot H^{2K+4}}^2 + C \| g \|_{L^2}^2.
        \label{5.2.2}
    \end{align}

 For the term $III$, following the argument in \eqref{linear estimate: nonlocal 1}, we first compute
 \begin{align*}
 & \quad  \int \left( \na \Delta^{K+2} Q \cdot \na \Dein g \right) \Delta^{K+2} g dy
   = \int \big| \partial_r \Delta^{K+2}Q \big| \left( \frac{1}{r^2} \int_0^r |g(s)| s^2 ds \right) \big| \Delta^{K+2} g \big| dy \\
  & \le \left( \int |\partial_r \Delta^{K+2} Q| |y|^{-\frac{1}{2}} |\Delta^{K+2} g| dy \right) \| g \|_{L^2} \le \Big\| |y|^{-\frac{1}{2}} \partial_r \Delta^{K+2} Q \Big\|_{L^2} \| g \|_{L^2} \| g \|_{\dot H^{2K+4}},
 \end{align*}
 which, by Lemma \ref{lem: solve ODE}, Gagliardo-Nirenberg inequality, and Young's inequality, yields that there exists $C=C(K,Q)>0$ such that
\begin{align}
    III
    & \le \int \left( \na \Delta^{K+2} Q \cdot \na \Dein g \right) \Delta^{K+2} g dy
    + C \left( \sum_{j=1}^{2K+4} \| \na^{2K+5-j} Q \|_{L^{\infty}} \| \na^{j+1} \Dein g \|_{L^2} \right) \| g \|_{\dot H^{2K+4}} \notag\\
    & \le C \| g \|_{H^{2K+3}} \| g \|_{\dot H^{2K+4}} 
    \le \frac{1}{100} \| g \|_{\dot H^{2K+4}}^2 + C \| g \|_{L^2}^2.
      \label{5.2.3}
\end{align}

Finally, for $IV$, by Lemma \ref{ODE system}, Gagliardo-Nirenberg inequality and Young's inequality again,
\begin{align}
    IV
    & \le  2(1-\mu) \int Q |\Delta^{K+2} g|^2 dy + C\sum_{j=0}^{2K+3} \| \na^{2K+4-j} Q \|_{L^\infty} \| \na^j g \|_{L^2} \| \Delta^{K+2} g \|_{L^2} \notag \\
    & \le  \left( 2(1-\mu) Q(0) + \frac{1}{100} \right) \| g \|_{\dot H^{2K+4}}^2 + C  \|g \|_{H^{2K+3}}^2. 
     \label{5.2.4}
\end{align}

    Consequently, combining \eqref{5.2.1}, \eqref{5.2.2}, \eqref{5.2.3}, \eqref{5.2.4}, along with the choice of $\beta < \frac{1}{2}$ from \eqref{beta: def}, Lemma \ref{lem: solve ODE}, the selection of $A$ from \eqref{A: choice range} and $K := \frac{A}{4}$, one finds that
    \begin{align*}
        & \quad \left( \calL g,g \right)_{\dot H^{2K+4}}
         \le  \Big( -\frac{\beta}{2} \left( 4K+5 \right)  +\frac{2K+4}{3(1-\mu)}+1+\frac{3}{100}\Big) \| g \|_{\dot H^{2K+4}}^2 + C\norm{g}_{L^2}^2 \\
        & =\Big(-\frac{A+8}{4j_0} + \frac{3\beta}{2}    +1+ \frac{3}{100} \Big)\| g \|_{\dot H^{2K+4}}^2 + C\norm{g}_{L^2}^2
         \le - \frac{1}{8} \| g \|_{\dot H^{2K+4}}^2 + C \| g \|_{L^2}^2,
    \end{align*}
    and thus, we have completed the proof of \eqref{coercivity in H^K}.
\end{proof}

\section{Nonlinear stability}
\label{sec: nonlinear sta}
In this section, building upon the coercivity properties of the operator $\calL$ established in Section \ref{sec: linear theory}, we proceed to analyze the nonlinear stability of $Q$. Then, based on this nonlinear stability, we will conclude this section by constructing a finite-time blowup solution to  \eqref{equation: KS with damping} for any $\mu \in \left[ 0,\frac{1}{3} \right)$ with finite mass and nonnegative density.

\subsection{Ansatz and modulation}\label{subsec:mod} 
Recalling the ansatz $\Psi = Q+ \ep$, where $\ep$ is \textit{radially symmetric} and solves the equation \eqref{linearized eq},  we further decompose $\ep$ as motivated by Proposition \ref{prop: coer}:
\be
\ep(\tau,y) = \epu(\tau,y) + \eps(\tau,y),
\text{ with } \epu(\tau,y) = \sum_{j=0}^K c_j(\tau) \chi(y) |y|^{2j},
\label{decomposition of Psi}
\ee
where $\chi$ is a radial smooth cut-off function on $B(0,1)$ defined in \eqref{def: cut-off function}, and $K = \frac{A}{4} \in \ZZ$ is as in Proposition \ref{prop: core H^K}. The coefficients $\{c_j \}_{j=0}^K$ are chosen such that the component $\eps$ satisfies the vanishing condition at the origin
\be
\partial_r^{2j} \ep_s (\tau,0) =0, \quad \; \forall \; 0 \le j \le K.
\label{stable part: modulation condition}
\ee
In particular, we split $\epu$ as
\be
\ep_u(\tau) = \ep_{u,real} + \ep_{u,fake} \;,
\label{Psi: decomposition}
\ee
where
\begin{equation}\label{epu: real & fake}
    \begin{cases}
        \ep_{u,real}(\tau,y) := \sum_{j=0}^{j_0} c_j(\tau) \chi(y) |y|^{2j},\\
    \ep_{u,fake}(\tau, y) := \sum_{j=j_0+1}^K c_{j} \chi(y) |y|^{2j}.
    \end{cases}
\end{equation}

\begin{remark}
   Recalling that the singular weight introduced in Proposition \ref{prop: coer} might not be sharp, the decomposition \eqref{decomposition of Psi} allows for stable components to be included in $\epu$. One can distinguish the unstable and stable parts as in \eqref{epu: real & fake}. For further intuition regarding decomposition \eqref{epu: real & fake}, we refer the readers to the modulation ODE system \eqref{ODE system for unstable part}.
\end{remark}

\subsubsection{Coefficients of the Taylor expansion of $\calL \epu$} With the vanishing condition \eqref{stable part: modulation condition}, one can derive the modulation equations for $\{c_j \}_{j=0}^K$. 
A crucial step in this process is to compute the Taylor expansion at the origin of each term appearing in \eqref{linearized eq}. As a preliminary, we focus in this section on determining the Taylor coefficients of $\calL \epu$ at the origin.

Specifically, recalling the definition of  $\calL$ from \eqref{linearized op}, and Taylor expansions of $(Q,f_Q)$ at the origin given in \eqref{(Q,f), expansion}, we obtain for any $r \in (0,\epsilon]$,
\begin{align*}
    & \quad \calL \epu 
     = - (\epu + \beta r \partial_r \epu) + f_Q \left(r \partial_r \epu \right) + r\partial_r Q \left( \frac{1}{r^3} \int_0^r \epu (s) s^2 ds \right) + 2(1-\mu) Q \epu \\
    & = -\sum_{j=0}^K c_j r^{2j}  - \beta  \sum_{j=0}^K 2j c_j r^{2j} + \left(\sum_{j=0}^\infty [f_Q]_{j} r^{2j} \right) \left( \sum_{j=0}^K 2j c_j r^{2j} \right) \\
  & \quad + \left(\sum_{j=0}^\infty 2j [Q]_{j} r^{2j} \right) \left( \sum_{j=0}^K  \frac{c_j}{2j+3} r^{2j} \right)  + 2(1-\mu) \left( \sum_{j=0}^\infty [Q]_{j}  r^{2j}\right) \left( \sum_{j=0}^K c_j r^{2j} \right).
\end{align*}
Applying \eqref{recurrence relation 1}, \eqref{recurrence Qj} and $[Q]_{j_0}=-1$, this yields that
\[
[\calL \epu]_j = -c_j - 2 \beta j c_j + 2j[f_Q]_0 c_j + 2(1-\mu) [Q]_0c_j 
= \frac{j_0 -j}{j_0} c_j, \qquad \forall \; 0 \le j < j_0,
\]
and
\begin{align*}
[\calL \ep]_{j_0} 
& = -c_{j_0} - 2 \beta j_0 c_{j_0} + 2j_0 [f_Q]_0 c_{j_0} + 2j_0 [Q]_{j_0} \frac{c_0}{3} + 2(1-\mu) [Q]_0 c_{j_0} + 2(1-\mu) [Q]_{j_0} c_0 \\
& 
=\sigma_{0,j_0} c_0, \quad \text{ where } \quad \sigma_{0,j_0} =  - \left( \frac{2}{3} j_0 + 2(1-\mu) \right).
\end{align*}
Similarly, for $j_0 +1 \le j \le K$, there exists $\{ \sigma_{i,j} \} =\{ \sigma_{i,j} (\mu,j_0) \} \subset \RR$ such that
\begin{align*}
    [\calL \ep]_{j}
    = \frac{j_0-j}{j_0} c_j + \sum_{i=0}^{j-1} \sigma_{i,j} c_i, \qquad \forall \; j_0+1 \le  j \le K.
\end{align*}
In summary, $[\calL \epu]_j$, the $2j$-th order coefficient of the Taylor expansion of $\calL \epu$ at the origin, satisfies
\be
[\calL \epu]_j
=
\begin{cases}
    \frac{j_0-j}{j_0} c_j, & 0 \le j < j_0, \\
    \frac{j_0-j}{j_0} c_j+ \sigma_{0,j_0} c_0, & j=j_0, \\
     \frac{j_0-j}{j_0} c_j + \sum_{i=0}^{j-1} \sigma_{i,j} c_i, & j_0+1 \le j \le K.
\end{cases}
\label{L epu: Taylor coefficients}
\ee

\subsubsection{Modulation equations.} Next, we derive the modulation equations for $\{c_j\}_{j=0}^K$ under the restriction given in \eqref{stable part: modulation condition}, building on the previous preliminary results. Concretely, with the decomposition of $\ep$ in \eqref{decomposition of Psi} and the equation \eqref{linearized eq},
\begin{align}
    \partial_\tau \eps
     & =  \calL \eps + N(\eps)  
     + \na \epu \cdot \na \Dein \eps + \na \eps \cdot \na \Dein \epu + 2(1-\mu) \epu \eps + G[\lambda, \Psi, \epu],
     \label{eq: eps}
\end{align}
where the \textit{modulation term} $G[\lambda,\Psi,\epu]$ is defined as
\be
G[\lambda, \Psi, \epu] := \lambda^{2-4\beta} \Delta \Psi - \sum_{j=0}^K \dot c_j \chi(r) r^{2j} + \calL \epu + N(\epu).
\label{G: def}
\ee
Since $\ep_s = O(r^{2K+2})$ at the origin as indicated in \eqref{stable part: modulation condition}, and by the definition of $\calL$ given in \eqref{linearized op}, $\calL \eps = O(r^{2K+2})$ at the origin, the following holds
\[
\partial_\tau \eps - \left( \calL \eps + N(\eps) + \na \epu \cdot \na \Dein \eps + \na \eps \cdot \na \Dein \epu + 2(1-\mu) \epu \eps \right) = O(r^{2K+2}),
\]
which requires the modulation term $G[\lambda,\Psi, \epu] = O(r^{2K+2})$ at the origin. Noting that $G[\lambda, \Psi, \epu]$ can be expanded as follows:
\begin{align}
    & \quad G[\lambda, \Psi ,\epu]
      = \lambda^{2-4 \beta} \Delta \Psi - \sum_{j=0}^K \dot c_j \chi r^{2j} + \calL \epu + N(\epu) \notag\\
    & = \lambda^{2-4\beta} \sum_{j=0}^K [\Delta \Psi ]_j \chi r^{2j} 
    - \sum_{j=0}^K \dot c_j \chi r^{2j} 
    + \sum_{j=0}^K [\calL \epu]_j \chi r^{2j} + \sum_{j=0}^K [N(\epu)]_j \chi  r^{2j}
    \label{G: lower term}
    \\
    & \quad + \lambda^{2-4\beta} \left( \left(\Delta Q + \epu \right) -\sum_{j=0}^K [\Delta Q + \Delta \epu ]_j \chi r^{2j} \right)
    + \lambda^{2-4\beta} \left( \Delta \eps  - [\Delta \eps]_K \chi r^{2K} \right) \notag \\
    & \quad 
    + \left( \calL \epu - \sum_{j=0}^K [\calL \epu]_j \chi r^{2j} \right) +N(\epu) -  \sum_{j=0}^K [N(\epu)]_j \chi  r^{2j} \notag,
\end{align}
where \eqref{G: lower term} has fewer orders of vanishing at the origin. Hence \eqref{G: lower term} should be imposed to vanish at the origin, then together with \eqref{L epu: Taylor coefficients}, this yields the following modulation equations onto the coefficients $\{c_j\}_{j=0}^K$:
\be
\begin{cases}
    \dot c_j =\frac{j_0 -j}{j_0} c_j + \lambda^{2-4 \beta} [\Delta \Psi]_j +  [N(\epu)]_{j}, & 0 \le j < j_0, \\
    \dot c_{j_0} =  \sigma_{0,j_0}c_0 +\lambda^{2-4 \beta} [\Delta \Psi]_{j_0}  +
    [N(\epu)]_j
    , & j =j_0, \\
    \dot c_{j} =  \frac{j_0 -j}{j_0} c_j + \sum_{i=0}^{j-1} \sigma_{i,j} c_i + \lambda^{2-4 \beta} [\Delta \Psi]_{j}  + [N(\epu)]_j, & j_0 < j \le K.
\end{cases}
\label{ODE system for unstable part}
\ee
In conclusion, we obtain an ODE-PDE system in terms of $\left(\{ c_j \}_{j=0}^K, \eps\right)$, which couples equations \eqref{eq: eps} and \eqref{ODE system for unstable part}. 
Moreover, the modulation term $G[\lambda, \Psi ,\epu]$ in \eqref{eq: eps} can be further expressed as follows:
\begin{align}
     G[\lambda, \Psi ,\epu]
    & =  \lambda^{2-4\beta} \left( \left(\Delta Q + \epu \right) -\sum_{j=0}^K [\Delta Q + \Delta \epu ]_j \chi r^{2j} \right) \notag
    + \lambda^{2-4\beta} \left( \Delta \eps  - [\Delta \eps]_K \chi r^{2K} \right) \\
    & \quad 
    + \left( \calL \epu - \sum_{j=0}^K [\calL \epu]_j \chi r^{2j} \right) +N(\epu) - \sum_{j=0}^K [N(\epu)]_j \chi r^{2j}.
    \label{G: new expression}
\end{align}

\subsection{Bootstrap argument.}
 We take $\delta_g \ll 1$ satisfying
    \be
    \delta_g \ll \min \Big\{ \frac{1}{8}, 2-4\beta, \frac{1}{j_0},  \frac{1}{|\sigma_{0,j_0}|}, A^{-1} , B  \Big\},
    \label{delta g: restriction}
    \ee
and introduce the function
\be
d_{real}(\tau)
:= \sqrt{ \frac{1}{2}\sum_{j=1}^{j_0-1} c_j^2 + \frac{10}{\delta_g^2} c_0^2 + \frac{1}{4 |\sigma_{0,j_0}|} c_{j_0}^2 },
\label{bootrap real unstable: norm}
\ee
which describes the behavior of $\ep_{u,real}$. The main result is as follows:
\begin{proposition}[Bootstrap]
\label{prop: bootstrap}
   There exists $0 < \delta_g \ll 1$ satisfying \eqref{delta g: restriction} such that there exists $(\delta_0,\delta_1, \delta_2, \delta_3,  \delta_{j_0+1}, \delta_{j_0+2} ..., \delta_K)$ with
    \be
    \delta_0  \ll \delta_3 \ll \delta_{j_0+1} \ll \delta_{j_0+2} \ll \delta_{j_0+3} \ll ... \ll \delta_{K} \ll \delta_1 \ll \delta_2 \ll \delta_3^\frac{1}{2} \ll \delta_g \ll Q(2),
    \label{delta i: restriction}
    \ee
    such that for any $\lambda(0) = \lambda_0$, $\ep_{s}(0) \in L_{w}^2 \cap \dot H_{rad}^{2K+6} (\RR^3)$ and $\{ c_j(0) \}_{j=j_0+1}^K$  \footnote{With this regularity assumption on initial data, together with \eqref{ss coordinte} and \eqref{ss renormalization}, by the standard fixed point arguments as in \cite{zbMATH01532872} or \cite[Appendix A]{ksnsblowup}, the local wellposedness can be ensured and the related solution to \eqref{equation: KS with damping} satisfying $\rho \in C_t^0 \left( [0,T_{max}), H_x^{2K+6}(\RR^3)  \right) \cap C_t^1\left([0,T_{max}), H_x^{2K+4}(\RR^3) \right)$.} satisfying
    \be
    |\lambda_0| + |\lambda_0|^{2-4\beta} + 
    \| \ep_{s}(0) \|_{ L_{w}^2 (\RR^3)} +  
    \| \ep_{s}(0) \|_{\dot H^{2K+4}(\RR^3)} + \sum_{j=j_0+1}^K |c_j(0)| \le \delta_0,
    \label{bootstrap: initial condition}
    \ee
    there exists $\{c_j(0) \}_{j=0}^{j_0}$ such that the radial solution to \eqref{renormalization KS} with initial datum
    \[
    (\Psi_0, \lambda_0) = \left (Q+ \sum_{j=0}^{j_0} c_j(0) \chi |y|^{2j} + \sum_{j=j_0+1}^K c_j(0) \chi |y|^{2j} + \ep_{s0}, \lambda_0 \right),
    \]
    globally exists and can be split into \eqref{decomposition of Psi} satisfying the following for all $\tau \ge 0$:

    \noindent $\bullet$ (Control the $L^2$ base norm of stable part $\eps$)
    \be 
    \| \ep_{s}(\tau) \|_{L_{w}^2} \le \delta_1 e^{- \frac{1}{2} \delta_g \tau}.
    \label{bootstrap assum: L2 stable}
    \ee

    \noindent $\bullet$ (Control the higher regularity of stable part $\eps$ via $\dot H^{2K+4}$ norm)
    \be
    \| \eps(\tau) \|_{\dot H^{2K+4}} \le \delta_2 e^{-\frac{1}{2} \delta_g \tau}.
    \label{bootstrap assum: Hm stable}
    \ee

    \noindent $\bullet$ (Control the real unstable part $\ep_{u,real}$)
    \be
    d_{real}(\tau) \le \delta_3 e^{-\frac{7}{10} \delta_g \tau}.
    \label{bootstrap assum: real unstable}
    \ee
    
    \noindent $\bullet$ (Control the fake unstable part $\ep_{u,fake}$)
    \be
    |c_j(\tau)| \le \delta_j e^{-\frac{7}{10} \delta_g \tau}, \quad \forall \; j_0+1 \le j \le K.
    \label{bootstrap assum: fake unstable}
    \ee
\end{proposition}

\vspace{0.5mm}

\begin{remark}
In what follows, the estimates \eqref{delta: choice 1}, \eqref{delta: choice 2}, \eqref{delta: choice 3}, and \eqref{outergoing-flux property} yield the requirements on $\{\delta_i\}$ summarized in \eqref{delta i: restriction}. 
Notably, the constants $C>0$ appearing below, which  might vary from line to line if necessary,  are all independent of $\{\delta_i\}$. Therefore,  we can choose appropriate $\{\delta_i\}$  in a manner to close Proposition \ref{prop: bootstrap}.
\end{remark}

\begin{remark}
Proposition \ref{prop: bootstrap} is the center of the paper. As in \cite{MR3986939,MerleblowupNLSdefocusing}, Proposition \ref{prop: bootstrap} will be proven via contradiction using the topological argument as follows: given $\left(\lambda_0, \ep_{u,fake}(0), \eps(0) \right)$ satisfying \eqref{bootstrap: initial condition}, we assume that for any $\ep_{u,real}(\tau,y) = \sum_{j=0}^{j_0} c_j(\tau) \chi(y) |y|^{2j}$ satisfying \eqref{bootstrap assum: real unstable} when $\tau=0$, the exit time
\be
\tau^* = \sup \; \Big\{ \tau \ge 0: \big(d_{real}, \{ c_j \}_{j={j_0+1}}^N, \eps \big) \text{ satisfy  \eqref{bootstrap assum: L2 stable}-\eqref{bootstrap assum: fake unstable} simultaneously on } [0,\tau] \Big\}
\label{exit time}
\ee
is finite. We then look for a contradiction under the assumptions of Proposition \ref{prop: bootstrap}. Subsequently, we study the flow satisfying \eqref{bootstrap assum: L2 stable}-\eqref{bootstrap assum: fake unstable} on $[0,\tau^*]$. Specifically, we will show that the bounds in \eqref{bootstrap assum: L2 stable}, \eqref{bootstrap assum: Hm stable}, and \eqref{bootstrap assum: fake unstable} can be further improved within the bootstrap regime. This implies that the only possible scenario for the solution to exit the bootstrap regime is for the real unstable condition \eqref{bootstrap assum: real unstable} to fail as $\tau> \tau^*$. Moreover, using the outer-going flux property of $d_{real}$ at the exit time $\tau = \tau^*$, we conclude a contradiction via Brouwer's fixed point theorem.
\end{remark}

\subsection{A priori estimates}

\subsubsection{Preliminaries}

In the section, as preliminaries, we are devoted to the estimate of $[\Delta \Psi]_j$ with $0 \le j \le K$, which is the coefficient of Taylor expansion of $\Delta \Psi$ with the $2j$-th order at the origin. And the main result is as follows:
\begin{lemma}
\label{lemma: preliminaries, Psi}
    Under the bootstrap assumption of Proposition \ref{prop: bootstrap}, as for the function $\Psi$ defined in \eqref{ss renormalization}, there exists a uniform constant $C=C(K)>0$ such that the quantity $[\Delta\Psi]_j:= \frac{1}{(2j)!} \partial_r^{2j} \Delta \Psi (0)$ satisfies
    \be
    \big| [\Delta\Psi]_j \big| \le
    \begin{cases}
        C(K), & 0 \le j < K, \\
        C(K) + C(K) \| \eps \|_{H^{2K+4}}, & j=K.
    \end{cases}
    \label{Delta Psi: estimate on Taylor coeff}
    \ee
\end{lemma}

\begin{proof}
Recall the decomposition of $\Psi$ defined in \eqref{decomposition of Psi},
\[
[\Delta \Psi]_j = \frac{1}{(2j)!} \partial_r^{2j} [\Delta \Psi] (0)
= \frac{1}{(2j)!} \partial_r^{2j} \Delta Q (0) + \frac{1}{(2j)!} \partial_r^{2j} \Delta \epu(0) + \frac{1}{(2j)!} \partial_r^{2j}\Delta \eps (0),
\]
where, by bootstrap assumptions \eqref{bootstrap assum: real unstable} and \eqref{bootstrap assum: fake unstable}, there exists a constant $C=C(K)>0$ such that
\[
\Bigg|\frac{1}{(2j)!} \partial_r^{2j} [\Delta Q] (0) \Bigg|+ \Bigg| \frac{1}{(2j)!} \partial_r^{2j} (\Delta \epu) (0) \Bigg| \le C(K), \quad \forall \; 0 \le j \le 2K.
\]
 
For $\partial_r^{2j} \Delta \eps(\tau,0)$, note that $\eps = O(r^{2K+2})$ at the origin,
\[
\partial_r^{2j} \Delta \eps(0) = 0, \quad \forall \; 0 \le j \le K-1.
\]
In addition, for $\partial_r^{2K} \Delta \eps(\tau,0)$, by Sobolev inequality, there exists $C=C(K)>0$ such that
\begin{align}
    |\partial_r^{2K} \Delta \eps(\tau,0)| & 
      \le C  \big| [\eps(\tau)]_{K+1} \big|
    \le C |\na^{2K+2} \eps (\tau,0)| \notag
    \\
    &  \le C\| \na^{2K+2} \eps(\tau) \|_{L^\infty} \le C \| \eps(\tau) \|_{H^{2K+4}}.
    \label{Taylor expansion: Delta eps}
\end{align}
Consequently, combining all of the arguments above immediately yields \eqref{Delta Psi: estimate on Taylor coeff}.
\end{proof}

\begin{lemma}[Estimates of terms related to $\epu$]
\label{lemma: preliminaries epu}
    Under the bootstrap assumptions of Proposition \ref{prop: bootstrap}, for the function $\epu$ defined in \eqref{decomposition of Psi}, $N(\epu)$ is compactly supported on $B(0,2)$ and satisfies
    \be
   \sum_{j=0}^K \Big| [N(\epu)(\tau)]_j \Big| + \Big\| N(\epu)(\tau) \Big\|_{H^{2K+4}} \le C \sum_{j=0}^K |c_j(\tau)|^2, \quad \forall \; \tau \in [0,\tau^*],
    \label{nonlinear term: epu}
    \ee
     for some uniform constant $C=C(K)>0$. In addition, $\calL \epu \in H^{2K+4}$ and satisfies
    \be
    \sum_{j=0}^K \Big| [\calL \epu(\tau)]_j \Big| + \big\| \calL \epu(\tau) \big\|_{H^{2K+4}} \le C \sum_{j=0}^K |c_j(\tau)|, \quad \forall \; \tau \in [0,\tau^*],
    \label{linear term: epu}
    \ee 
    for some uniform constant $C=C(K,Q)>0$.
\end{lemma}

\begin{proof}
Recalling the definition of $\epu$ given in \eqref{decomposition of Psi}, and $N(\epu) = \na \Dein \epu \cdot \na \epu + (1-\mu) \epu^2$, we check
\begin{align*}
   N(\epu) 
    &= (r \partial_r \epu) \frac{1}{r^3} \int_0^r \epu(s) s^2 ds 
    + (1-\mu) \epu^2  \\
    & = \sum_{j_1,j_2=0}^K \left(\frac{2j_1 c_{j_1} c_{j_2}}{2j_2+3}
    + (1-\mu) c_{j_1} c_{j_2}  \right) r^{2j_1+2j_2} \\
    & = \sum_{j=0}^{2K} \left(  \sum_{j_1+j_2 =j}\frac{2j_1 c_{j_1} c_{j_2}}{2j_2+3} + (1-\mu) c_{j_1} c_{j_2}\right) r^{2j},  \quad \forall \; r \le \frac{1}{8}.
\end{align*}
This yields 
\begin{align*}
    \Big| [N(\epu)]_i \Big|
    & =  \Bigg| \sum_{j_1+j_2 =j}\frac{2j_1 c_{j_1} c_{j_2}}{2j_2+3} + (1-\mu) c_{j_1} c_{j_2} \Bigg| 
    \le C(K) \sum_{j=0}^K |c_j|^2, \quad \forall \; 0 \le i \le K.
\end{align*}
In addition, by using Gagliardo-Nirenberg inequality and  $\text{supp } \epu \in B(0,2)$,
\begin{align*}
\| N(\epu) \|_{H^{2K+4}}
& \le C | \epu \|_{H^{2K+5}}^2
\le C \sum_{j=0}^{K} |c_j|^2,
\end{align*}
which concludes \eqref{nonlinear term: epu}.
To verify \eqref{linear term: epu}, it is a direct result of \eqref{L epu: Taylor coefficients} that there exists a constant $C=C(K,Q)>0$ such that
\[
\Big| [\calL \epu]_i \Big| \le C(K,Q) \sum_{j=0}^K |c_j|, \quad \forall \; 0 \le i \le K.
\]
Additionally, recall \eqref{nonlocal term: formula}, Lemma \ref{lem: solve ODE} and $\text{supp } \epu \subset B(0,2)$, similar to the argument in \eqref{5.2.3} and \eqref{5.2.4}, by using Gagliardo-Nirenberg inequality, there exists 
a constant $C=C(K,Q)>0$ such that
\begin{align*}
    \big\| \calL \epu \big\|_{H^{2K+4}}  
    & \le  \Big\| -\left( \epu + \beta y \cdot \na \epu  \right) + \na \epu \cdot \na \Dein Q + 2(1-\mu) Q \epu \Big\|_{H^{2K+4}(B(0,2))} \\
    &  \quad + \| \na \Dein \epu \cdot \na Q \|_{H^{2K+4}} \\
    & \le C \| Q \|_{H^{2K+5}}\| \epu \|_{H^{2K+5}} \le C \sum_{j=0}^K |c_j|.
\end{align*}
This completes the proof of \eqref{linear term: epu}.
\end{proof}

\

\subsubsection{$L_w^2$ estimate of $\eps$.}
This section is to improve the $L_w^2(\RR^3)$ estimate of $\eps$ under the bootstrap assumptions of Proposition \ref{prop: bootstrap}. And the main result is as follows:
\begin{lemma}[$L_w^2$ estimate of $\eps$]
\label{lemma: L2 estimate of eps}
    Under the assumptions of Proposition \ref{prop: bootstrap}, there exists a constant $C=C(Q, A,B,\delta_g)>0$ such that
    \begin{align}
        \frac{d}{d\tau} \| \eps \|_{L_w^2}^2
        \le 
     - \delta_g \| \eps \|_{L_w^2}^2  +  C \left(\lambda^{4-8\beta} +  d_{real}^2 +\sum_{j=j_0+1}^K |c_j|^2  \right) \quad \forall \; \tau \in [0,\tau^*].
    \label{L2 estimate of eps}
    \end{align}
    In particular, the bootstrap assumption \eqref{bootstrap assum: L2 stable} can be improved to
    \[
    \| \eps (\tau) \|_{L_w^2} \le \frac{1}{2} \delta_1 e^{-\frac{\delta_g}{2} \tau}, \quad \forall \; \tau \in [0,\tau^*].
    \]
\end{lemma}

The proof of Lemma \ref{lemma: L2 estimate of eps} is lengthy; thus, we separate the $L_w^2$ estimate of the modulation term $G[\lambda, \Psi, \epu]$ defined in \eqref{G: def}. The estimate is as follows:
\begin{lemma}[Estimate on modulation term \text{$G[\lambda,\Psi,\epu]$}]
\label{estimate of G}
    Under the bootstrap assumption of Proposition \ref{prop: bootstrap}, as for $G[\lambda, \Psi, \epu]$ defined in \eqref{G: def},
    there exists a uniform constant $C=C(Q,A,B)>0$ such that
    \be
    \| G [\lambda, \Psi, \epu] \|_{L_w^2} \le C \lambda^{2-4\beta}  + C \sum_{j=0}^K |c_j|.
    \label{estimate of error term: singular weight}
    \ee
\end{lemma}
\begin{proof}[Proof of Lemma \text{\ref{estimate of G}}]
Recall \eqref{G: new expression}, $G[\lambda,\Psi, \epu] = O(r^{2K+2})$ at the origin, then by applying
 Taylor expansion with the form of integral residue at the origin onto $G[\lambda,\Psi,\epu]$,
\begin{align*}
     \big \| G[\lambda, \Psi, \epu] \big \|_{L_w^2(|y| \le 1)}^2
     & \le C \int_{|y| \le 1} |G[\lambda,\Psi, \epu]|^2 \frac{1}{|y|^{4K}} dy  \\
     & \le C \big\| D^{2K} G[\lambda, \Psi, \epu] \big\|_{L^\infty}^2 \le 
     C \| G[\lambda, \Psi, \epu] \|_{H^{2K+2}}^2.
\end{align*}
For the case with $|y| \ge 1$,
\begin{align*}
    \big \| G[\lambda, \Psi, \epu] \big \|_{L_w^2(|y| \ge 1)}
    & \le C \big \| G[\lambda, \Psi, \epu] \big \|_{L^2(|y| \ge 1)}  \le C  \| G[\lambda, \Psi, \epu] \|_{H^{2K+2}}.
\end{align*}

Consequently, together with the definition of $G[\lambda, \Psi,\epu]$ in \eqref{G: new expression}, Lemma \ref{lemma: preliminaries, Psi}, \eqref{Taylor expansion: Delta eps}, and Lemma \ref{lemma: preliminaries epu}, this concludes that
\begin{align*}
  & \quad  \| G[\lambda, \Psi, \epu] \|_{L_w^2} \le 
     C \| G[\lambda, \Psi, \epu] \|_{H^{2K+2}} \\
     & \le C\lambda^{2-4\beta} \left( \| \Delta Q + \epu \|_{H^{2K+2}} +\| \Delta \eps \|_{H^{2K+2}}+ \sum_{j=0}^K \big|[\Delta Q + \epu]_j \big| + [\Delta \eps]_K \right) \\
    &   \qquad +C \| \calL \epu \|_{H^{2K+2}} + \sum_{j=0}^K \big|[\calL \epu]_j \big | + C \| N(\epu) \|_{H^{2K+2}} + C \sum_{j=0}^K \big|[N(\epu)]_j \big| \\
    & \le C \lambda^{2-4\beta} \left(1+ \| \eps \|_{H^{2K+4}} \right) + C \left( \sum_{j=0}^K |c_j| + |c_j|^2 \right)
    \le C  \lambda^{2-4\beta}  + C  \sum_{j=0}^K |c_j|.
\end{align*}
for some uniform constant $C>0$, and thus we have concluded the proof.
\end{proof}

 Next, we are devoted to the proof of Lemma \ref{lemma: L2 estimate of eps}.
 
\begin{proof}[Proof of Lemma \ref{lemma: L2 estimate of eps}]
Recall that $\eps$ solves the equation \eqref{eq: eps},
    \begin{align*}
        \frac{1}{2} \frac{d}{d\tau} \| \eps \|_{L_w^2}^2
          = (\partial_\tau \eps, \eps)_{L_w^2} 
          = I + I I + III + IV,
    \end{align*}
    with
    \begin{align*}
        & I: = \left( \calL \eps, \eps \right)_{L_w^2}, \quad 
         II:= \left(\na \epu \cdot \na \Dein \eps + \na \eps \cdot \na \Dein \epu + 2(1-\mu) \epu \eps, \eps \right)_{L_w^2}, \\
        & III
        := \left( N(\eps), \eps \right)_{L_w^2},
        \quad  IV:= \left(G[\lambda, \Psi, \epu], \eps \right)_{L_w^2}.
    \end{align*}

    \noindent \textit{Step 1. Estimate of the small linear term $II$.} The small linear term $II$ can be written as $II:=II_1+II_2+ II_3$ with
    \[
    II_1 : =\left(\na \epu \cdot \na \Dein \eps , \eps \right)_{L_w^2}, \quad 
    II_2 := \left(\na \eps \cdot \na \Dein \epu, \eps \right)_{L_w^2}, \quad 
    II_3:= \left( 2(1-\mu) \epu \eps, \eps \right)_{L_w^2}.
    \]
    First of all, repeating \eqref{I NLO 1}, together with the fact that $\text{supp }\epu \subset B(0,2)$, there exists $C= C(A)>0$ such that $II_1$ can be controlled by
\begin{align*}
    II_1 \le C \big\| |y|^{-\frac12} \partial_r \epu \big\|_{L^2} \| \eps \|_{L_w^2}^2 \le C \| \na \epu \|_{L^\infty} \| \eps \|_{L_w^2}^2.
\end{align*}

For $II_2$, direct computation shows
\begin{align}
    II_2 
    & = \int \left( \na \eps \cdot \na \Dein \epu \right) \eps w dy  = \frac{1}{2} \int \na \eps^2 \cdot \na \Dein \epu w dy \notag \\
    & = -\frac{1}{2} \int \eps^2  \epu w dy - \frac{1}{2} \int \eps^2 \na \Dein \epu \cdot \na w dy \le C \| \epu \|_{L^\infty} \| \eps \|_{L_w^2}^2,
    \label{L2: weight norm II 2}
\end{align}
where the last inequality holds by Cauchy-Schwarz inequality together with the pointwise estimate
\begin{align*}
    |\na \Dein \epu \cdot \na w |
    & = \Bigg| \left( \frac{1}{r^2} \int_0^r \epu (s) s^2 ds \right) \left( \partial_r w  \right) \Bigg|
    \le C(A) \| \epu \|_{L^\infty} w.
\end{align*}

For $II_3$, it can be directly controlled by
\begin{align*}
    II_3 & \le \Big| \left( 2(1-\mu) \epu \eps , \eps \right)_{L_w^2} \Big|
     \le 2(1-\mu) \| \epu \|_{L^\infty} \| \eps \|_{L_w^2}^2.
\end{align*}

In sum, the small linear term $II$ can be estimated by
\begin{align*}
 II \le C \| \epu\|_{W^{1,\infty}} 
 \| \eps\|_{L_w^2}^2.
\end{align*}

\vspace{0.3cm}
\noindent \textit{Step 2. Estimate of the nonlinear term $III$.} 
We begin by performing integration by parts and using a similar argument in \eqref{L2: weight norm II 2}, it then follows that there exists $C=C(A)>0$ such that
\begin{align*}
    III & =\left( \na \eps \cdot \na \Dein \eps + (1-\mu) \eps^2, \eps  \right)_{L_w^2}  \le C \| \eps \|_{L^\infty} \| \eps \|_{L_w^2}^2.
\end{align*}

\vspace{0.3cm}

\noindent \textit{Step 3. Summary of the estimates.} Combining all of the estimates above, \eqref{L: coer}, \eqref{delta g: restriction}, \eqref{delta i: restriction} and \eqref{estimate of error term: singular weight}, and applying Young's inequality, we obtain that there exists a constant $C= C(Q,A,B,\delta_g) >0$ such that
\begin{align*}
    \frac{1}{2} \frac{d}{d\tau} \| \eps \|_{L_w^2}^2
    & \le -\frac{1}{8} \| \eps \|_{L_w^2}^2 + C \left( \| \epu \|_{W^{1,\infty}} + \| \eps \|_{L^\infty} \right) \| \eps \|_{L_w^2}^2 + C \left( \lambda^{2-4\beta} +\sum_{j=0}^{K} |c_j| \right) \| \eps \|_{L_w^2} \\
    & \le - \frac{\delta_g}{2} \| \eps \|_{L_w^2}^2 + C \left( \lambda^{4-8\beta} + \sum_{j=0}^K |c_j|^2 \right),
\end{align*}
which concludes \eqref{L2 estimate of eps}. Hence, recalling the bootstrap assumptions \eqref{bootstrap assum: L2 stable}, \eqref{bootstrap assum: Hm stable}, \eqref{bootstrap assum: real unstable} and \eqref{bootstrap assum: fake unstable}, and by Gronwall's inequality,
\begin{align}
    & \quad \| \eps(\tau) \|_{L_w^2}^2 
    \le  e^{-\delta_g \tau} \| \eps(0) \|_{L_w^2}^2
    + C \int_0^\tau e^{-\delta_g (\tau -s)} \left(  \lambda^{4-8\beta}(s) + d_{real}^2(s) + \sum_{j=j_0+1}^K |c_j(s)|^2 \right) ds \notag \\
    & \le  e^{-\delta_g \tau} \| \eps(0) \|_{L_w^2}^2 + C \int_0^\tau e^{-\delta_g (\tau -s)} \left( \lambda_0^{4-8\beta} e^{-(2-4\beta) s} + \left( \delta_3^2 + \sum_{j=j_0+1}^K \delta_j^2 \right) e^{-\frac{7}{5} \delta_g s} \right) ds \notag \\
    & \le \left[ \delta_0^2 + C \left( \lambda_0^{4-8\beta} + \delta_3^2 + \sum_{j=j_0+1}^K \delta_j^2 \right)  \right] e^{-\delta_g \tau}, \quad \forall \; \tau \in (0,\tau^*], 
    \label{delta: choice 1}
\end{align}
for some constant $C=C(Q, A,B,\delta_g)>0$ independent on each $\delta_i$. Ttogether with \eqref{delta i: restriction} and \eqref{bootstrap: initial condition}, this yields that
\[
\| \eps (\tau) \|_{L_w^2} \le \frac{1}{2} \delta_1 e^{-\frac{1}{2} \delta_g \tau}, \quad \forall \; \tau \in [0,\tau^*],
\]
hence we have concluded the proof.
\end{proof}

\subsubsection{$\dot H^{2K+4}$ estimate of $\eps$.} This section is to improve the $\dot H^{2K+4}(\RR^3)$ estimate of $\eps$ under the bootstrap assumptions of Proposition \ref{prop: bootstrap}. The main result is as follows:

\begin{lemma}[$\dot H^{2K+4}$ estimate of $\eps$]
\label{lemma: Hm estimate of eps}
Under the assumptions of Proposition \ref{prop: bootstrap}, there exists a constant $C=C(Q, A,B,\delta_g)>0$ such that the following inequality hold uniformly in $\tau \in [0,\tau^*]$:
\be
\frac{d}{d\tau} \| \eps \|_{\dot H^{2K+4}}^2 \le -  \frac{1}{16} \| \eps \|_{\dot H^{2K+4}}^2 + C \| \eps \|_{L^2}^2
    +C\left( \| \eps \|_{H^{2K+4}}^4 + \lambda^{4-8\beta} + \sum_{j=0}^{K} |c_j|^2 \right).
\label{stable part: higher regu}
\ee
In particular, the bootstrap assumption \eqref{bootstrap assum: Hm stable} can be improved to
\[
\| \eps (\tau) \|_{\dot H^{2K+4}} \le \frac{1}{2} \delta_2 e^{-\frac{\delta_g}{2} \tau}, \quad \forall \; \tau \in [0,\tau^*].
\]
\end{lemma}

\begin{proof}
    Recall that $\eps$ solves the equation \eqref{eq: eps},
    \begin{align*}
        &  \frac{1}{2} \frac{d}{d\tau} \| \eps \|_{\dot H^{2K+4}}^2
         = I + II + III + IV,
    \end{align*}
    with
    \begin{align*}
         &  I : = \left ( \Delta^{K+2}\calL \eps,\Delta^{K+2} \eps\right )_{L^2}, \\
         & II:= \left( \Delta^{K+2} \left( \na \epu \cdot \na \Dein \eps + \na \eps \cdot \na \Dein \epu + 2(1-\mu) \epu \eps \right), \Delta^{K+2} \eps \right)_{L^2}, \\
         &III:= \left(\Delta^{K+2}N(\eps),\Delta^{K+2} \eps\right )_{L^2}, \quad   
         IV:= \left(\Delta^{K+2}G[\lambda, \Psi, \epu],\Delta^{K+2} \eps\right )_{L^2}.
    \end{align*}

    \noindent \underline{\textit{Estimate of small linear term $II$.}} For the small linear term $II$, since $\epu \in C_0^\infty(\RR^3)$ with support in $B(0,2)$, together with integration by parts, Sobolev inequality, Gagliardo-Nirenberg inequality and \eqref{nonlocal term: formula}, there exists uniform constant $C=C(K,\mu)>0$ such that
    \begin{align}
         II
        & \le  \left( \Delta^{K+2} \left( \na \eps \cdot \na \Dein \epu \right), \Delta^{K+2} \eps \right)_{L^2}
        + \left( \Delta^{K+2} \left( \na \epu \cdot \na \Dein \eps + 2(1-\mu) \epu \eps \right), \Delta^{K+2} \eps \right)_{L^2} \notag \\
        & \le \left( \na \Delta^{K+2} \eps \cdot \na \Dein \epu, \Delta^{K+2} \eps \right)_{L^2}
        + C \sum_{j=1}^{2K+4} \| \na^{2K+5-j} \eps \|_{L^2} \| \na^{j+1} \Dein \epu \|_{L^\infty} \| \eps \|_{\dot H^{2K+4}}  \notag \\
        & \quad +C \| \na^{2K+5} \epu \|_{L^2} \| \na \Dein \eps \|_{L^\infty(B(0,2))}
        + C \sum_{j=1}^{2K+4}\| \na^{2K+5-j} \epu \|_{L^\infty} \| \na^{j+1} \Dein \eps \|_{L^2} \notag \\
        & \le C \| \epu \|_{H^{2K+6}} \| \eps \|_{H^{2K+4}}^2.
         \label{Hk: III}
    \end{align}

    \noindent \underline{\textit{Estimate of nonlinear term $III$.}} For $III$, by integration by parts, there exists a uniform constant $C=C(K,\mu)>0$ such that
    \begin{align}
        III
        & = \left(\Delta^{K+2} \left( \na \eps \cdot \na \Dein \eps \right),\Delta^{K+2} \eps\right )_{L^2} + (1-\mu)\left(\Delta^{K+2} (\eps^2),\Delta^{K+2} \eps\right )_{L^2} \notag \\
        & \le  C  \left( \| \eps \|_{L^\infty} + \| \na^2 \Dein \eps \|_{L^\infty}  \right)  \| \eps \|_{H^{2K+4}}^2
        +  C \sum_{j=2}^{2K+4} \| \eps \|_{H^{2K+6-j}} \| \eps \|_{H^{j}} \| \eps \|_{\dot H^{2K+4}}
        \notag
        \\
        & \quad +C \sum_{j=1}^{2K+3} \| \eps \|_{H^{2K+5-j}} \| \eps \|_{H^{j+1}} \| \eps \|_{\dot H^{2K+4}} \notag\\
        & \le C \| \eps \|_{H^{2K+4}}^3.
        \label{HK: II}
    \end{align}
    
    \noindent \underline{\textit{Estimate of the modulation term $IV$.}} For the modulation term $IV$, recalling \eqref{G: new expression}, \eqref{Taylor expansion: Delta eps}, \eqref{nonlinear term: epu} and \eqref{linear term: epu},
    there exists a uniform constant $C>0$ such that
    \begin{align*}
        \Big\| G[\lambda,\Psi, \epu] - \lambda^{2-4\beta} \Delta \eps \Big\|_{\dot H^{2K+4}}
        \le C \lambda^{2-4\beta} + C \sum_{j=0}^K |c_j|,
    \end{align*}
    which, by the non-negativity of $\lambda$ (see \eqref{lambda: form}), yields that
    \begin{align}
        IV & = \left( \Delta^{K+2} G[\lambda,\Psi, \epu], \Delta^{K+2} \eps \right)_{L^2} \notag \\
        & =  \lambda^{2-4 \beta} \left( \Delta^{K+2} \Delta \eps, \Delta^{K+2} \eps \right)_{L^2} + \left( \Delta^{K+2} \left( G[\lambda,\Psi, \epu]-\lambda^{2-4\beta} \Delta \eps \right), \Delta^{K+2} \eps \right)_{L^2}
        \notag
        \\
        & = - \lambda^{2-4\beta} \| \na \Delta^{K+2} \eps \|_{L^2} +
        \left( \Delta^{K+2} \left( G[\lambda,\Psi, \epu]-\lambda^{2-4\beta} \Delta \eps \right), \Delta^{K+2} \eps \right)_{L^2} 
        \notag
        \\
        & \le C \left( \lambda^{2-4\beta} +  \sum_{j=0}^K |c_j| \right) \| \eps \|_{\dot H^{2K+4}}.
        \label{HK: IV}
    \end{align}

    \noindent \underline{\textit{Conclusion of the estimates.}} Noting that $I$ has been handled in Proposition \ref{prop: core H^K}, combining \eqref{coercivity in H^K}, \eqref{HK: II}, \eqref{Hk: III} and \eqref{HK: IV} concludes that
    \begin{align*}
    \frac{d}{d\tau} \frac{1}{2} \| \eps \|_{\dot H^{2K+4}}^2
    & \le - \frac{1}{8} \| \eps \|_{\dot H^{2K+4}}^2 + C \| \eps \|_{L^2}^2 \\
    & \quad \;+C\left( \| \eps \|_{H^{2K+4}}^2  + \lambda^{2-4\beta} + \sum_{j=0}^{K} |c_j| \right) \| \eps \|_{H^{2K+4}},
    \end{align*}
    which, together with Young's inequality, yields \eqref{stable part: higher regu}. Hence, recalling the bootstrap assumptions \eqref{bootstrap assum: L2 stable}, \eqref{bootstrap assum: Hm stable}, \eqref{bootstrap assum: real unstable} and \eqref{bootstrap assum: fake unstable}, by Gronwall's inequality,
    \begin{align}
        & \quad \| \eps(\tau) \|_{\dot H^{2K+4}}^2 \notag \\
        & \le \| \eps(0) \|_{\dot H^{2K+4}}^2 e^{-\frac{1}{16} \tau}
        +C \left(\delta_1^2 + \delta_2^4 + \delta_3^2 + \delta_0^{4-8\beta} + \sum_{j=0}^K \delta_j^2  \right) \int_0^\tau e^{-\frac{1}{16} (\tau-s)} e^{-\delta_g s} ds \notag \\
        & \le C \left(\delta_0^2 +\delta_1^2 + \delta_2^4 + \delta_3^2 + \delta_0^{4-8\beta} + \sum_{j=0}^K \delta_j^2  \right)  e^{-\delta_g \tau}, \quad \forall \; \tau \in [0,\tau^*].
        \label{delta: choice 2}
    \end{align}
    Combining the assumptions on the coefficients \eqref{delta g: restriction} and \eqref{delta i: restriction}, this yields that
    \[
    \| \eps(\tau) \|_{\dot H^{2K+4}} \le \frac{1}{2} \delta_2 e^{-\frac{1}{2} \delta_g \tau}, \qquad \forall \; \tau \in [0,\tau^*],
    \]
    and hence we have concluded the result.
\end{proof}

\subsubsection{Estimate of $\ep_{u,fake}$.}
This section is dedicated to improving the bootstrap assumption \eqref{bootstrap assum: fake unstable} under the bootstrap assumptions of Proposition \ref{prop: bootstrap}. The main result is stated as follows:

\begin{lemma}[Estimate of $\ep_{u,fake}$]
\label{lemma: estimate of unstable fake}
    Under the assumptions of Proposition \ref{prop: bootstrap}, for any $j_0+1 \le j \le K$, there exists a constant $C=C(\mu, K,j_0, \beta)>0$ such that
    \begin{align}
        \frac{d}{d\tau} c_j^2 \le - \frac{1}{j_0} c_j^2 +   C \lambda^{8 -4\beta} + C \sum_{i=0}^{j-1} |c_i|^2 + C \sum_{i=0}^K |c_i|^4, \quad \forall \; \tau \in [0,\tau^*].
        \label{fake unstable: estimate}
    \end{align}
    In particular, the bootstrap assumption \eqref{bootstrap assum: fake unstable} can be improved to
    \be
    |c_j (\tau)| \le \frac{1}{2} \delta_j e^{-\frac{7\delta_g}{10} \tau}, \quad \forall \; \tau \in [0,\tau^*].
    \label{fake unstable: improvement of bootstrap}
    \ee
\end{lemma}

\begin{proof}
Recalling \eqref{ODE system for unstable part}, $c_j$ solves the ODE
\[
\dot c_{j} =  \frac{j_0 -j}{j_0} c_j + \lambda^{2-4 \beta} [\Delta \Psi]_{j} +  \sum_{i=0}^{j-1} \sigma_{i,j} c_i + [N(\epu)]_{j}, \quad \text{  for any } j_0 < j \le K.
\]
Then by Young's inequality, \eqref{Delta Psi: estimate on Taylor coeff} and \eqref{nonlinear term: epu}, there exists $C=C(K, j_0, \mu)>0$ such that
\begin{align*}
    \frac{1}{2}\frac{d}{d\tau}  c_j^2 
     = c_j \dot c_j &= c_j \left( \frac{j_0 -j}{j_0} c_j + \lambda^{2-4 \beta} [\Delta \Psi]_{j} + \sum_{i=0}^{j-1} \sigma_{i,j} c_i + [N(\epu)]_j \right) \\
    & = \frac{j_0-j}{j_0} c_j^2 + \lambda^{2-4\beta}  [\Delta \Psi]_j c_j + \sum_{i=0}^{j-1} \sigma_{i,j} c_i c_j + [N(\epu)]_j c_j \\
    & \le - \frac{1}{2j_0} c_j^2 +  C \lambda^{8 -4\beta} + C \sum_{i=0}^{j-1} |c_i|^2 + C \sum_{i=0}^K |c_i|^4.
\end{align*}

By Gronwall's inequality, \eqref{delta g: restriction}  \eqref{delta i: restriction},  \eqref{bootstrap: initial condition}, and bootstrap assumption \eqref{bootstrap assum: real unstable} and \eqref{bootstrap assum: fake unstable}, 
there exists a constant $C= C(K, j_0,\beta)>0$ such that
\begin{align}
    c_j^2(\tau) 
    & \le e^{-\frac{1}{j_0} \tau} c_{j}^2(0) 
    + C \int_0^\tau e^{-\frac{1}{j_0} (\tau-s)} \Bigg[ \lambda_0^{4-8\beta} e^{-(2-4\beta) s}  + 
    \left( \delta_3^2+
    \sum_{i=j_0+1}^{j-1} \delta_i^2 \right) e^{-\frac{7}{5} \delta_g s}  \notag \\
    & \qquad \qquad \qquad \qquad \qquad \qquad \qquad \qquad \qquad + \left( \delta_3^4 + \sum_{i=j_0+1}^K \delta_i^4 \right) e^{-\frac{14}{5} \delta_g s} \Bigg] ds \notag \\
    & \le  e^{-\frac{1}{j_0} \tau} c_{j}^2(0) + C \left( \lambda_0^{4-8\beta} + \delta_3^2 + \sum_{j_0+1 \le i \le j-1} \delta_i^2  + \delta_3^4 + \sum_{i=j_0+1}^K \delta_i^4 \right) e^{-\frac{7}{5} \delta_g \tau}
    \notag
    \\
    & \le C \left( \delta_0^2 + \lambda_0^{4-8\beta} + \delta_3^2 + \sum_{i=j_0+1}^{j-1} \delta_i^2 \right) e^{-\frac{7}{5}\delta_g \tau} \le \frac{1}{4} \delta_j^2 e^{-\frac{7}{5} \delta_g \tau}, \;\; \forall \; \tau \in [0,\tau^*],
    \label{delta: choice 3}
\end{align}
which concludes the proof.
\end{proof}

\subsection{Control of real unstable part.}

\mbox{}

\begin{proof}[Proof of Proposition \ref{prop: bootstrap}]
    
    \mbox{}

    \noindent \underline{\textit{Improvement of the bootstrap assumptions.}} Arguing by contradiction, suppose that there exists an initial triple $(\lambda_0, \ep_{u,fake}(0), \eps(0))$ satisfying \eqref{bootstrap: initial condition} such that for any $\ep_{u,fake}(0)$ with $d_{real}(0) \le \delta_3$, the exit time $\tau^*$ defined in \eqref{exit time} is always finite. Then, by Lemma \ref{lemma: L2 estimate of eps}, Lemma \ref{lemma: Hm estimate of eps} and Lemma \ref{lemma: estimate of unstable fake}, combining with the continuity of the solution to system \eqref{renormalization KS} with respect to $\tau$, there exists $0 < \epsilon_{\tau^*} \ll 1$ such that \eqref{bootstrap assum: L2 stable}, \eqref{bootstrap assum: Hm stable} and \eqref{bootstrap assum: fake unstable} hold for all $\tau \in [0,\tau^* + \epsilon_{\tau^*}]$.
    
\vspace{0.3cm}

\noindent \underline{\textit{Outer-going flux property.}}
In light of the preceding argument, the only possible mechanism by which the solution exits the bootstrap regime is the failure of the bootstrap assumption for the real unstable component $\ep_{u,real}$, specifically when this assumption ceases to hold on $[0,\tau]$ for some $\tau > \tau^*$. Moreover, by the local well-posedness of the system \eqref{equation: KS with damping}, it follows that the mapping $\tau \mapsto d_{real}(\tau)$ is continuous. If we define the set
\[
B(\tau) := \Big\{ \{c_j\}_{j=0}^{j_0} \in \RR^{j_0+1}: \sqrt{ \frac{1}{2}\sum_{j=1}^{j_0-1} c_j^2 + \frac{10}{\delta_g^2} c_0^2 + \frac{1}{4 |\sigma_{0,j_0}| } c_{j_0}^2} \le \delta_3 e^{-\frac{7}{10} \delta_g \tau} \Big\},
\]
then by the contradiction assumption, the following holds
\[
\begin{cases}
    \{ c_j(\tau) \}_{j=0}^{j_0} \in B(\tau), & \forall \; \tau \in [0,\tau^*], \\
    \{ c_j(\tau) \}_{j=0}^{j_0} \in \partial B(\tau), & \tau = \tau^*.
\end{cases}
\]
 Recalling the modulation ODE system satisfied by the coefficients $\{c_j\}_{j=0}^{j_0}$ as given in \eqref{ODE system for unstable part}, and invoking Young's inequality in conjunction with Lemma \ref{lemma: preliminaries, Psi} and Lemma \ref{lemma: preliminaries epu}, there exists a constant $C=C(\delta_g, |\sigma_{0,j_0}|,K,j_0)>0$ independent of $\{\delta_i\}$, such that the quantity $d_{real}$ defined in \eqref{bootrap real unstable: norm} satisfies
\begin{align*}
    & \quad \frac{d}{d\tau}d_{real}^2(\tau)
    = \sum_{j=1}^{j_0-1} c_j \dot c_j + \frac{20}{\delta_g^2} c_0 \dot c_0+ \frac{1}{2 |\sigma_{0,j_0}|} c_{j_0} \dot c_{j_0}  \\
    & = \sum_{j=1}^{j_0-1} \frac{j_0-j}{j_0} c_j^2 + \frac{20}{\delta_g^2} c_0^2 + \frac{1}{2} c_0 c_{j_0}  + \lambda^{2-4\beta}\left( \sum_{j=1}^{j_0-1} c_j  [\Delta \Psi]_j + \frac{20}{\delta_g^2} c_0 [\Delta \Psi]_0 + \frac{1}{2 |\sigma_{0,j_0}|}c_{j_0} [\Delta \Psi]_{j_0} \right) \\
    & \qquad + \sum_{j=1}^{j_0 -1} [N(\epu)]_j c_j + \frac{20}{\delta_g^2} c_0 [N(\epu)]_0 + \frac{1}{2 |\sigma_{0,j_0}|} c_{j_0} [N(\epu)]_{j_0} \\
    & \ge - C \lambda^{4-8\beta}\sum_{j=0}^{j_0} \Big|[\Delta \Psi]_j \Big|^2 - \frac{\delta_g^2}{160} c_{j_0}^2 - C \sum_{j=0}^K \big| [N(\epu)]_{j} \big|^2 \\
    &\ge - C \lambda^{4-8\beta}
    - \frac{\delta_g^2 |\sigma_{0,j_0}|}{40} d_{real}^2
     - Cd_{real}^4  - C \sum_{j=j_0+1}^K |c_j|^4.
\end{align*}
Consequently, under the contradiction assumption that $e^{\frac{7}{10}\delta_g \tau^*}d_{real}(\tau^*) = \delta_3$, 
and in light of bootstrap assumptions \eqref{bootstrap: initial condition}, \eqref{bootstrap assum: real unstable}, and \eqref{bootstrap assum: fake unstable}, together with coefficient restrictions \eqref{delta g: restriction} and \eqref{delta i: restriction}, we deduce the outer-going flux property on $\partial B(\tau^*)$:
\begin{align}
    & \quad \frac{d}{d\tau} \left( e^{\frac{7}{5} \delta_g \tau} d_{real}^2(\tau) \right) \Bigg|_{\tau = \tau^*}
    = \left( \frac{7}{5} \delta_g e^{\frac{7}{5} \delta_g \tau} d_{real}^2(\tau) + e^{\frac{7}{5} \delta_g \tau} \frac{d}{d\tau} d_{real}^2(\tau) \right) \Big|_{\tau = \tau^*} \notag\\
    & \ge  \frac{7}{5} \delta_g e^{\frac{7}{5} \delta_g \tau^*} d_{real}^2(\tau^*)
    - C \lambda_0^{4-8\beta}e^{\frac{7}{5} \delta_g \tau^*} e^{-(2-4\beta) \tau^*} - \frac{\delta_g^2 |\sigma_{0,j_0}|}{40} e^{\frac{7}{5} \delta_g \tau^*} d_{real}^2(\tau^*) \notag\\
    & \qquad \qquad 
    -C e^{\frac{7}{5} \delta_g \tau^*} d_{real}^4(\tau^*) -  C \sum_{j=j_0+1}^K e^{\frac{7}{5} \delta_g \tau^*} |c_j(\tau^*)|^4 \notag
    \\
    & \ge \frac{7}{5} \delta_g \delta_3^2 - C \lambda_0^{4-8\beta} - \frac{\delta_g^2 |\sigma_{0,j_0}|}{40} \delta_3^2 
    -C \delta_3^4 - C \sum_{j=j_0+1}^K \delta_j^4
    \ge \delta_g \delta_3^2>0.
    \label{outergoing-flux property}
\end{align}

\vspace{0.3cm}
\noindent \underline{\textit{Brouwer's topological argument.}} 
Firstly, for any $\delta >0$, we define the following convex subset on $\RR^{j_0+1}$:
\[
\calB_{real}(0,\delta)
:= \left\{ \{ c_j \}_{j=0}^{j_0} \in \RR^{j_0+1}:  \sqrt{ \frac{1}{2}\sum_{j=1}^{j_0-1} c_j^2 + \frac{10}{\delta_g^2} c_0^2 + \frac{1}{4 |\sigma_{0,j_0}|} c_{j_0}^2} < \delta  \right\}.
\]
By the local well-posedness theory for \eqref{equation: KS with damping} and contradiction assumption \eqref{exit time}, 
we establish the continuity of the map $\ep_{u,real} (0) \mapsto \ep_{u,real}(\tau^*(\ep_{u,real}(0)))$, where $\tau^*(\ep_{u,real}(0))$ denotes the exit time associated with 
$\left( \lambda_0, \ep_{u,real}(0), \ep_{u,fake}(0), \eps(0) \right)$ satisfying the initial condition \eqref{bootstrap: initial condition}. Accordingly, we define a mapping $\Phi: \overline{\calB_{real} (0,1)} \mapsto \partial \calB_{real} (0,1) $ as follows:
\begin{align*}
 \qquad  \{d_j\}_{j=0}^{j_0} \in \overline{\calB_{real} (0,1)}
&\mapsto \{c_{j}(0)\}_{j=0}^{j_0} := \{ \delta_3 d_j\}_{j=0}^{j_0}  \in \overline{\calB_{\RR^{j_0+1}} (0,\delta_3)} \\
&  \mapsto  \ep_{u,real}(0,y) =  \sum_{j=0}^{j_0} c_j(0) \chi(y) |y|^{2j}  \\
& \mapsto \ep_{u,real}(\tau^*,y) =\sum_{j=0}^{j_0} c_j(\tau^*) \chi(y) |y|^{2j} \\
&
 \mapsto \{ d_j(\tau^*) \}_{j=0}^{j_0} := \Big\{  \frac{1}{\delta_3}e^{\frac{7}{10} \delta_g \tau^*}c_j(\tau^*) \Big\}_{j=0}^{j_0} \in  \partial \calB_{real}(0,1),
\end{align*}
whose continuity is ensured by the local well-posedness of \eqref{equation: KS with damping}. In particular, when $\{d_j\}_{j=0}^{j_0} \in \partial \calB_{real} (0,1)$, by the outer-going flux property of the flow as indicated in \eqref{outergoing-flux property}, $\tau =0$ is exactly the exit time, and it leads to $\Phi = Id$ on the boundary $\partial \calB_{real}(0,1)$. However, with the continuity of $\Phi$ on the nonempty compact convex set $\overline{\calB_{real} (0,1)}$, we conclude that $-\Phi$ has a fixed point on the boundary $\partial \calB_{real}(0,1)$, which contradicts to the assertion that $\Phi =Id$ on the boundary. Consequently, the proof of Proposition \ref{prop: bootstrap} has been completed.
\end{proof}

\subsection{Existence of finite blowup solution to \eqref{equation: KS with damping} with finite mass.} This section is devoted to the study of the existence of the finite-time blowup solution to \eqref{equation: KS with damping} for any fixed $0 \le \mu < \frac{1}{3}$ and concludes the paper.

\begin{proof}[Proof of \textcolor{black}{Theorem \ref{thm: existence of blowup}}]
    Firstly, for any fixed $\mu \in \left[0,\frac{1}{3} \right)$, we choose $Q$ constructed in Lemma \ref{lem: solve ODE} and fix the weight $w$ as specified in Proposition \ref{prop: coer}. Then as indicated in Proposition \ref{prop: bootstrap}, there exist $\delta_g \ll 1$ and $\{ \delta_i \}$ satisfying \eqref{delta g: restriction} and \eqref{delta i: restriction}, such that the result shown in Proposition \ref{prop: bootstrap} holds.
    
    \noindent \underline{\textit{Construction of nonnegative initial data with finite mass.}}
    Firstly, since $Q\in H^{2K+4}$ as guaranteed by Lemma \ref{lem: solve ODE}, we can choose $\lambda_0 \ll \delta_0$ sufficiently small and $R_2 \gg 1$ sufficiently large such that the following estimates hold:
    \be
    \| (\chi_{R_2} -1) Q \|_{L_w^2} + \| (\chi_{R_2} -1) Q \|_{\dot H^{2K+4}} \ll \delta_0,
    \label{cut off: smallness}
    \ee
    where $\chi_{R_2}$ is a cut-off function on $B(0,R_2)$ defined by \eqref{def: cut-off on B(0,R)}.
    It then follows that there exists $\{ c_j(0) \}_{j=0}^{j_0}$ such that Proposition \ref{prop: bootstrap} holds with the initial datum
    \be
    \left(\lambda_0, \ep_s(0,y), \ep_{u,real}(0,y), \ep_{u,fake}(0,y) \right)
    := \left(\lambda_0, (\chi_{R_2} -1) Q, \sum_{j=0}^{j_0} c_j(0) \chi(y) |y|^{2j} ,0\right).
    \label{blowup to KS: initial data}
    \ee
    In other words, there exists $\ep_{u,real}(0,y)$ such that the solution to \eqref{renormalization KS} with initial data 
    \begin{align*}
    \Psi_0 &= Q + \eps(0) + \ep_{u,real}(0) + \ep_{u,fake}(0)
    =\chi_{R_2}Q + \sum_{j=0}^{j_0} c_{j}(0) \chi(y) |y|^{2j} \subset C_0^\infty(\RR^3),
    \end{align*}
    globally exists in $\tau$ and satisfies \eqref{bootstrap assum: L2 stable}-\eqref{bootstrap assum: fake unstable} on $[0,+\infty)$, and the corresponding solution $\Psi(\tau)$ can be decomposed into $\Psi(\tau) = Q + \ep(\tau)$ with
    \begin{align*}
        \| \ep(\tau) \|_{H^{2K+4}}
         & \le  \| \epu(\tau) \|_{H^{2K+4}} + \| \eps(\tau) \|_{H^{2K+4}} \\
        & \le C \left( \delta_1 + \delta_2 + \delta_3 + \sum_{j=j_0+1}^K \delta_j \right) e^{-\frac{1}{2} \delta_g \tau}, \quad \forall \; \tau >0,
    \end{align*}
    for some universal $C>0$. In addition, in view of the fact that $Q$ is radially decreasing, together with the condition \eqref{delta i: restriction}, we obtain that
    \[
    \begin{cases}
     \Psi_0(y) \ge Q(2) - C \sum_{j=0}^{j_0}|c_j(0)| \ge \frac{1}{2} Q(2) >0, &  \forall \; |y| \le 2, 
     \vspace{0.2cm}
     \\
    \Psi_0(y) = \chi_{R_2}(y)Q(y) \ge 0, & \forall \; |y| \ge 2,
    \end{cases}
    \]
    which yields the non-negativity of $\Psi_0$. Now, returning to the original equation \eqref{equation: KS with damping}, and recalling \eqref{ss coordinte}, \eqref{ss renormalization} and \eqref{lambda: form}, 
    the nonnegativity of initial density $\rho_0$ can be guaranteed:
    \[
    \rho_0(x) = \frac{1}{\lambda_0^2} \Psi_0\left( \frac{x}{\lambda_0^{2\beta}} \right) \in C_0^\infty(\RR^3).
    \]

    \noindent \underline{\textit{Finite time blowup.}} Recalling \eqref{ss coordinte}, we know that
    \begin{align*}
        \frac{d\lambda}{dt} = \frac{d\lambda}{d\tau} \frac{d\tau}{dt}
        = -\frac{\lambda}{2} \frac{1}{\lambda^2} = - \frac{1}{2\lambda}, \quad \text{with } \lambda(0) = \lambda_0,
    \end{align*}
    which implies that $\lambda(t)$ blows up at $T=\lambda_0^2$ and $\lambda(t)$ is given by
    \[
    \lambda(t) = \sqrt{T-t} = \sqrt{\lambda_0^2 -t}, \quad \forall \; t \in [0,T).
    \]
    Consequently, the solution to \eqref{equation: KS with damping} with initial data $ \rho_0$ is of form
    \[
    \rho(t,x) = \frac{1}{T-t} \left( Q + \e \right)\left(t, \frac{x}{(T-t)^{\beta}}\right),
    \]
    with
    \[
    \| \ep(\tau) \|_{H^{2K+4}} \le \epsilon_0 (T-t)^{\frac{\delta_g}{2}},
    \]
    for some $0 < \epsilon_0 \ll 1$. Consequently, we have concluded the proof of Theorem \ref{thm: existence of blowup}.

\end{proof}

\normalem
\bibliographystyle{siam}
\bibliography{Bib-1}

\begin{thebibliography}{10}

\bibitem{MR4201903}
{\sc P.~Biler}, {\em Singularities of solutions to chemotaxis systems}, vol.~6
  of De Gruyter Series in Mathematics and Life Sciences, De Gruyter, Berlin,
  [2020] \copyright 2020.

\bibitem{MR3411100}
{\sc P.~Biler, I.~Guerra, and G.~Karch}, {\em Large global-in-time solutions of
  the parabolic-parabolic {K}eller-{S}egel system on the plane}, Commun. Pure
  Appl. Anal., 14 (2015), pp.~2117--2126.

\bibitem{Blanchet_Dolbeault_Perthame_globalexistence06}
{\sc A.~Blanchet, J.~Dolbeault, and B.~Perthame}, {\em Two-dimensional
  {K}eller-{S}egel model: optimal critical mass and qualitative properties of
  the solutions}, Electron. J. Differential Equations,  (2006), pp.~No. 44, 32.

\bibitem{Brenner_Constantin_Leo_Schenkel_Venkataramani_steady_state_99}
{\sc M.~P. Brenner, P.~Constantin, L.~P. Kadanoff, A.~Schenkel, and S.~C.
  Venkataramani}, {\em Diffusion, attraction and collapse}, Nonlinearity, 12
  (1999), pp.~1071--1098.

\bibitem{bricmont1994universality}
{\sc J.~Bricmont and A.~Kupiainen}, {\em Universality in blow-up for nonlinear
  heat equations}, Nonlinearity, 7 (1994), p.~539.

\bibitem{Javierblowup3DNS}
{\sc T.~Buckmaster, G.~Cao-Labora, and J.~G{\'o}mez-Serrano}, {\em Smooth
  imploding solutions for 3{D} compressible fluids}, Forum Math. Pi, 13 (2025),
  p.~139.
\newblock Id/No e6.

\bibitem{buseghin2023existence}
{\sc F.~Buseghin, J.~Davila, M.~del Pino, and M.~Musso}, {\em {E}xistence of
  finite time blow-up in {K}eller-{S}egel system}, arXiv preprint
  arXiv:2312.01475,  (2023).

\bibitem{nonradialblowupNS}
{\sc G.~Cao-Labora, J.~G{\'o}mez-Serrano, J.~Shi, and G.~Staffilani}, {\em
  Non-radial implosion for compressible {Euler} and {Navier}-{Stokes} in {$\Bbb
  T^d$} and {$\Bbb R^d$}}.
\newblock Preprint, {arXiv}:2310.05325 [math.{AP}] (2023), 2023.

\bibitem{nonradialblowupNLS}
\leavevmode\vrule height 2pt depth -1.6pt width 23pt, {\em Non-radial implosion
  for the defocusing nonlinear {Schr{\"o}dinger} equation in {$\Bbb T^d$} and
  {$\Bbb R^d$}}.
\newblock Preprint, {arXiv}:2410.04532 [math.{AP}] (2024), 2024.

\bibitem{MR3932458}
{\sc J.~A. Carrillo, K.~Craig, and Y.~Yao}, {\em Aggregation-diffusion
  equations: dynamics, asymptotics, and singular limits}, in Active particles.
  {V}ol. 2. {A}dvances in theory, models, and applications, Model. Simul. Sci.
  Eng. Technol., Birkh\"auser/Springer, Cham, 2019, pp.~65--108.

\bibitem{chen2024vorticityhigherdimension}
{\sc J.~Chen}, {\em Vorticity blowup in compressible {Euler} equations in
  {$\Bbb R^d$}, {$d \ge 3$}}.
\newblock Preprint, {arXiv}:2408.04319 [math.{AP}] (2024), 2024.

\bibitem{chen2024vorticity}
{\sc J.~Chen, G.~Cialdea, S.~Shkoller, and V.~Vicol}, {\em Vorticity blowup in
  2{D} compressible {E}uler equations}, arXiv preprint arXiv:2407.06455,
  (2024).

\bibitem{Chen_Hou_Euler_blowup_theory_2023}
{\sc J.~Chen and T.~Y. Hou}, {\em Stable nearly self-similar blowup of the 2{D}
  {B}oussinesq and 3{D} {E}uler equations with smooth data {I}: {A}nalysis},
  preprint, arXiv:2210.07191,  (2023).

\bibitem{Chen_Hou_Euler_blowup_numerical_2023}
\leavevmode\vrule height 2pt depth -1.6pt width 23pt, {\em Stable nearly
  self-similar blowup of the 2{D} {B}oussinesq and 3{D} {E}uler equations with
  smooth data {II}: {R}igorous {N}umerics}, preprint, arXiv:2305.05660,
  (2023).

\bibitem{Chen_Hou_Huang_De_Gregorio_eq_2021}
{\sc J.~Chen, T.~Y. Hou, and D.~Huang}, {\em On the finite time blowup of the
  {D}e {G}regorio model for the 3{D} {E}uler equations}, Comm. Pure Appl.
  Math., 74 (2021), pp.~1282--1350.

\bibitem{chen2024stability}
{\sc J.~Chen, T.~Y. Hou, V.~T. Nguyen, and Y.~Wang}, {\em On the stability of
  blowup solutions to the complex {G}inzburg-{L}andau equation in {$\Bbb
  R^d$}}, arXiv preprint arXiv:2407.15812,  (2024).

\bibitem{Collot_Ghoul_Masmoudi_Nguyen_2DtypeII_blowup22}
{\sc C.~Collot, T.-E. Ghoul, N.~Masmoudi, and V.~T. Nguyen}, {\em Refined
  description and stability for singular solutions of the 2{D} {K}eller-{S}egel
  system}, Comm. Pure Appl. Math., 75 (2022), pp.~1419--1516.

\bibitem{Collot_Ghoul_Masmoudi_Nguyen_3Dblowup_Collasping-ring_blowup23}
\leavevmode\vrule height 2pt depth -1.6pt width 23pt, {\em Collapsing-ring
  blowup solutions for the {K}eller-{S}egel system in three dimensions and
  higher}, J. Funct. Anal., 285 (2023), pp.~Paper No. 110065, 41.

\bibitem{2DKSmultibubble}
\leavevmode\vrule height 2pt depth -1.6pt width 23pt, {\em Singularity formed
  by the collision of two collapsing solitons in interaction for the 2{D}
  {Keller}-{Segel} system}.
\newblock Preprint, {arXiv}:2409.05363 [math.{AP}] (2024), 2024.

\bibitem{MR3986939}
{\sc C.~Collot, P.~Rapha\"{e}l, and J.~Szeftel}, {\em On the stability of type
  {I} blow up for the energy super critical heat equation}, Mem. Amer. Math.
  Soc., 260 (2019), pp.~v+97.

\bibitem{MR2099126}
{\sc L.~Corrias, B.~Perthame, and H.~Zaag}, {\em Global solutions of some
  chemotaxis and angiogenesis systems in high space dimensions}, Milan J.
  Math., 72 (2004), pp.~1--28.

\bibitem{Dolbeault_Perthame_globalexistence04}
{\sc J.~Dolbeault and B.~Perthame}, {\em Optimal critical mass in the
  two-dimensional {K}eller-{S}egel model in {$\Bbb R^2$}}, C. R. Math. Acad.
  Sci. Paris, 339 (2004), pp.~611--616.

\bibitem{MR3537340}
{\sc R.~Donninger and B.~Sch\"{o}rkhuber}, {\em On blowup in supercritical wave
  equations}, Comm. Math. Phys., 346 (2016), pp.~907--943.

\bibitem{MR4334974}
{\sc T.~Elgindi}, {\em Finite-time singularity formation for {$C^{1,\alpha}$}
  solutions to the incompressible {E}uler equations on {$\Bbb R^3$}}, Ann. of
  Math. (2), 194 (2021), pp.~647--727.

\bibitem{MR4445341}
{\sc T.~M. Elgindi, T.-E. Ghoul, and N.~Masmoudi}, {\em On the stability of
  self-similar blow-up for {$C^{1,\alpha}$} solutions to the incompressible
  {E}uler equations on {$\Bbb R^3$}}, Camb. J. Math., 9 (2021), pp.~1035--1075.

\bibitem{engel2000one}
{\sc K.-J. Engel and R.~Nagel}, {\em One-parameter semigroups for linear
  evolution equations}, vol.~194 of Graduate Texts in Mathematics,
  Springer-Verlag, New York, 2000.

\bibitem{Fuest_blowup_optimal_2021}
{\sc M.~Fuest}, {\em Approaching optimality in blow-up results for
  {K}eller-{S}egel systems with logistic-type dampening}, NoDEA Nonlinear
  Differential Equations Appl., 28 (2021), pp.~Paper No. 16, 17.

\bibitem{stable3DKS}
{\sc I.~Glogi\'{c} and B.~Sch\"{o}rkhuber}, {\em Stable singularity formation
  for the {Keller}-{Segel} system in three dimensions}, Arch. Ration. Mech.
  Anal., 248 (2024), p.~40.
\newblock Id/No 4.

\bibitem{ghj21}
{\sc Y.~Guo, M.~Had{\v{z}}i{\'c}, and J.~Jang}, {\em Larson-penston
  self-similar gravitational collapse}, Commun. Math. Phys., 386 (2021),
  pp.~1551--1601.

\bibitem{guoMahirJangpolytropicgap}
{\sc Y.~Guo, M.~Had{\v{z}}i{\'c}, J.~Jang, and M.~Schrecker}, {\em
  Gravitational collapse for polytropic gaseous stars: self-similar solutions},
  Arch. Ration. Mech. Anal., 246 (2022), pp.~957--1066.

\bibitem{MR1651769}
{\sc M.~A. Herrero, E.~Medina, and J.~J.~L. Vel\'azquez}, {\em Self-similar
  blow-up for a reaction-diffusion system}, J. Comput. Appl. Math., 97 (1998),
  pp.~99--119.

\bibitem{MR2448428}
{\sc T.~Hillen and K.~J. Painter}, {\em A user's guide to {PDE} models for
  chemotaxis}, J. Math. Biol., 58 (2009), pp.~183--217.

\bibitem{MR2013508}
{\sc D.~Horstmann}, {\em From 1970 until present: the {K}eller-{S}egel model in
  chemotaxis and its consequences. {I}}, Jahresber. Deutsch. Math.-Verein., 105
  (2003), pp.~103--165.

\bibitem{MR2073515}
\leavevmode\vrule height 2pt depth -1.6pt width 23pt, {\em From 1970 until
  present: the {K}eller-{S}egel model in chemotaxis and its consequences.
  {II}}, Jahresber. Deutsch. Math.-Verein., 106 (2004), pp.~51--69.

\bibitem{hou2025axisymmetric}
{\sc T.~Y. Hou, V.~T. Nguyen, and P.~Song}, {\em Axisymmetric type {II} blowup
  solutions to the three dimensional keller-segel system}, arXiv preprint
  arXiv:2502.19775,  (2025).

\bibitem{hou20242}
{\sc T.~Y. Hou, V.~T. Nguyen, and Y.~Wang}, {\em {$L^2$}-based stability of
  blowup with log correction for semilinear heat equation}, arXiv preprint
  arXiv:2404.09410,  (2024).

\bibitem{hou2024blowup}
{\sc T.~Y. Hou and Y.~Wang}, {\em Blowup analysis for a quasi-exact 1{D} model
  of 3{D} euler and {N}avier--{S}tokes}, Nonlinearity, 37 (2024), p.~035001.

\bibitem{jls2025}
{\sc J.~Jang, J.~Liu, and M.~Schrecker}, {\em Converging/diverging self-similar
  shock waves: From collapse to reflection}, SIAM Journal on Mathematical
  Analysis, 57 (2025), pp.~190--232.

\bibitem{jang2023selfsimilarconvergingshockwaves}
\leavevmode\vrule height 2pt depth -1.6pt width 23pt, {\em On self-similar
  converging shock waves}, Arch. Ration. Mech. Anal., 249 (2025).

\bibitem{jiaSverakillposedness}
{\sc H.~Jia and V.~Sverak}, {\em Are the incompressible 3d {N}avier-{S}tokes
  equations locally ill-posed in the natural energy space?}, J. Funct. Anal.,
  268 (2015), pp.~3734--3766.

\bibitem{Kang_Stevens_globalexistence_p=2_and_mu_critical}
{\sc K.~Kang and A.~Stevens}, {\em Blowup and global solutions in a
  chemotaxis-growth system}, Nonlinear Anal., 135 (2016), pp.~57--72.

\bibitem{kim2022self}
{\sc J.~Kim}, {\em Self-similar blow up for energy supercritical semilinear
  wave equation}, arXiv preprint arXiv:2211.13699,  (2022).

\bibitem{Limodestability}
{\sc Z.~Li}, {\em Mode stability for self-similar blowup of {$L^2$} slightly
  supercritical {NLS}}, in preparation.

\bibitem{ksnsblowup}
{\sc Z.~Li and T.~Zhou}, {\em Finite-time blowup for
  {Keller}-{Segel}-{Navier}-{Stokes} system in three dimensions}.
\newblock Preprint, {arXiv}:2404.17228 [math.{AP}] (2024), 2024.

\bibitem{lizhou2025nonradial}
\leavevmode\vrule height 2pt depth -1.6pt width 23pt, {\em Nonradial stability
  of self-similar blowup to {K}eller-{S}egel equation in three dimensions},
  arXiv preprint arXiv:2501.07073,  (2025).

\bibitem{MerleblowupNLSdefocusing}
{\sc F.~Merle, P.~Rapha\"{e}l, I.~Rodnianski, and J.~Szeftel}, {\em On blow up
  for the energy super critical defocusing nonlinear {Schr{\"o}dinger}
  equations}, Invent. Math., 227 (2022), pp.~247--413.

\bibitem{MR4445442}
\leavevmode\vrule height 2pt depth -1.6pt width 23pt, {\em On the implosion of
  a compressible fluid {I}: {S}mooth self-similar inviscid profiles}, Ann. of
  Math. (2), 196 (2022), pp.~567--778.

\bibitem{MR4445443}
\leavevmode\vrule height 2pt depth -1.6pt width 23pt, {\em On the implosion of
  a compressible fluid {II}: {S}ingularity formation}, Ann. of Math. (2), 196
  (2022), pp.~779--889.

\bibitem{MR2729284}
{\sc F.~Merle, P.~Rapha\"{e}l, and J.~Szeftel}, {\em Stable self-similar
  blow-up dynamics for slightly {$L^2$} super-critical {NLS} equations}, Geom.
  Funct. Anal., 20 (2010), pp.~1028--1071.

\bibitem{MR1427848}
{\sc F.~Merle and H.~Zaag}, {\em Stability of the blow-up profile for equations
  of the type {$u_t=\Delta u+|u|^{p-1}u$}}, Duke Math. J., 86 (1997),
  pp.~143--195.

\bibitem{MR1361006}
{\sc T.~Nagai}, {\em Blow-up of radially symmetric solutions to a chemotaxis
  system}, Adv. Math. Sci. Appl., 5 (1995), pp.~581--601.

\bibitem{zbMATH01532872}
\leavevmode\vrule height 2pt depth -1.6pt width 23pt, {\em Behavior of
  solutions to a parabolic-elliptic system modelling chemotaxis}, J. Korean
  Math. Soc., 37 (2000), pp.~721--733.

\bibitem{Naito_Suzuki_typeIIblowup}
{\sc Y.~Naito and T.~Suzuki}, {\em Self-similarity in chemotaxis systems},
  Colloq. Math., 111 (2008), pp.~11--34.

\bibitem{nguyen2023construction}
{\sc V.~Nguyen, N.~Nouaili, and H.~Zaag}, {\em Construction of type {I}-{L}og
  blowup for the {K}eller-{S}egel system in dimensions $3$ and $4$}, to appear
  in Annal of PDE,  (2023).
\newblock Available at arXiv:2309.13932.

\bibitem{nguyen2025infinitely}
{\sc V.~T. Nguyen, Z.-A. Wang, and K.~Zhang}, {\em Infinitely many self-similar
  blow-up profiles for the {K}eller-{S}egel system in dimensions 3 to 9}, arXiv
  preprint arXiv:2503.02263,  (2025).

\bibitem{MR3438649}
{\sc T.~Ogawa and H.~Wakui}, {\em Non-uniform bound and finite time blow up for
  solutions to a drift--diffusion equation in higher dimensions}, Anal. Appl.
  (Singap.), 14 (2016), pp.~145--183.

\bibitem{MR4752990}
{\sc S.-J. Oh and F.~Pasqualotto}, {\em Gradient blow-up for dispersive and
  dissipative perturbations of the {B}urgers equation}, Arch. Ration. Mech.
  Anal., 248 (2024), pp.~Paper No. 54, 61.

\bibitem{MR1788983}
{\sc K.~J. Painter, P.~K. Maini, and H.~G. Othmer}, {\em Development and
  applications of a model for cellular response to multiple chemotactic cues},
  J. Math. Biol., 41 (2000), pp.~285--314.

\bibitem{Raphael_Schweyer_2DtypeII_blowup14}
{\sc P.~Rapha\"{e}l and R.~Schweyer}, {\em On the stability of critical
  chemotactic aggregation}, Math. Ann., 359 (2014), pp.~267--377.

\bibitem{MR540951}
{\sc N.~Shigesada, K.~Kawasaki, and E.~Teramoto}, {\em Spatial segregation of
  interacting species}, J. Theoret. Biol., 79 (1979), pp.~83--99.

\bibitem{MR3936129}
{\sc P.~Souplet and M.~Winkler}, {\em Blow-up profiles for the
  parabolic-elliptic {K}eller-{S}egel system in dimensions {$n\geq 3$}}, Comm.
  Math. Phys., 367 (2019), pp.~665--681.

\bibitem{Tello_Winkler_globalexi_for_p=2}
{\sc J.~I. Tello and M.~Winkler}, {\em A chemotaxis system with logistic
  source}, Comm. Partial Differential Equations, 32 (2007), pp.~849--877.

\bibitem{Winkler_blowup_N>=5_2011}
{\sc M.~Winkler}, {\em Blow-up in a higher-dimensional chemotaxis system
  despite logistic growth restriction}, J. Math. Anal. Appl., 384 (2011),
  pp.~261--272.

\bibitem{Winkler_blowup_N=3_or_4_2018}
\leavevmode\vrule height 2pt depth -1.6pt width 23pt, {\em Finite-time blow-up
  in low-dimensional {K}eller-{S}egel systems with logistic-type superlinear
  degradation}, Z. Angew. Math. Phys., 69 (2018), pp.~Paper No. 69, 40.

\end{thebibliography}

\end{document}